\documentclass[11pt,a4paper]{article}

\usepackage{titlesec}
\usepackage{fancyhdr}
\usepackage{a4wide}
\usepackage{graphicx}
\usepackage{float}
\usepackage{amssymb}
\usepackage{amsmath}
\usepackage{amsthm}
\usepackage{color}
\usepackage{mathrsfs}
\usepackage{array}
\usepackage{eucal}
\usepackage{tikz}
\usetikzlibrary{calc}
\usepackage[T1]{fontenc}
\usepackage{inputenc}
\usepackage[english]{babel}
\usepackage{lmodern}
\usepackage{hyperref}
\usepackage{geometry}
\usepackage{changepage}

\changepage{0pt}{}{}{}{}{0pt}{}{0pt}{5pt}
\usepackage[numbers]{natbib}
\usepackage{hyperref}
\hypersetup{
pdfpagemode=none,
pdftoolbar=true,        
pdfmenubar=true,        
pdffitwindow=false,     
pdfstartview={Fit},    
pdftitle={Inside-outside duality with artificial backgrounds},    
pdfkeywords={}, 
pdfnewwindow=true,      
colorlinks=true,       
linkcolor=magenta,          
citecolor=red,        
filecolor=cyan,      
urlcolor=blue           
}

\newcommand{\dsp}{\displaystyle}

\newcommand{\Om}{\Omega}
\newcommand{\rel}{\#}
\newcommand{\N}{\mathbb{N}}
\newcommand{\R}{\mathbb{R}}
\newcommand{\mL}{\mrm{L}}
\newcommand{\mH}{\mrm{H}}

\newcommand{\mX}{\mrm{X}}
\newcommand{\mY}{\mrm{Y}}

\newcommand{\loc}{\mbox{\scriptsize loc}}
\newtheorem{theorem}{Theorem}[section]

\newtheorem{remark}{Remark}[section]

\newtheorem{proposition}{Proposition}[section]

\def\R{\mathbb R}
\def\N{\mathbb N}
\def\Om{\Omega}

\newcommand{\mrm}[1]{\mathrm{#1}}

\newcommand{\eps}{\varepsilon}

\DeclareMathOperator{\supp}{\mathrm{supp}}
\newcommand{\Cplx}{\mathbb{C}}

\newcommand{\herg}{H}

\begin{document}

~\vspace{0.4cm}
\begin{center}
{\sc \bf\LARGE 
\scalebox{1.04}{Inside-outside duality with artificial backgrounds}}\\[15pt]
\end{center}

\begin{center}
\textsc{Lorenzo Audibert}$^{1,2}$, \textsc{Lucas Chesnel}$^2$, \textsc{Houssem Haddar}$^{2}$\\[18pt]
\begin{minipage}{0.9\textwidth}
{\small
$^1$ Department PRISME, EDF R\&D, 6 quai Watier, 78401, Chatou CEDEX, France;\\[2pt]
$^2$ INRIA,  Ecole Polytechnique, CMAP, Route de Saclay, 91128 Palaiseau, France.\\[10pt]
E-mails: \scalebox{0.9}{\texttt{lorenzo.audibert@edf.fr}, \texttt{lucas.chesnel@inria.fr}, \texttt{houssem.haddar@inria.fr}}\\[-14pt]
\begin{center}
(\today)
\end{center}
}
\end{minipage}
\end{center}
\vspace{0.4cm}

\noindent\textbf{Abstract.} 
We use the inside-outside duality approach proposed by Kirsch-Lechleiter to identify transmission eigenvalues associated with artificial backgrounds. We prove that for well chosen artificial backgrounds, in particular for the ones with zero index of refraction at the inclusion location, one obtains a necessary and sufficient condition characterizing transmission eigenvalues via the spectrum of the modified far field operator. We also complement the existing literature with a convergence result for the invisible generalized incident field associated with the transmission eigenvalues.
\\
This work is based on several of the pioneering works of our dearest colleague and friend Armin Lechleiter and is dedicated to his memory. 
\\[6pt]

\section{Introduction}

The present work is motivated by the problem of retrieving information on the material index $n$ of a penetrable inclusion embedded in a reference medium from far field data associated with incident plane waves. Some recent works have suggested the use of so-called Transmission Eigenvalues (TEs) as qualitative spectral signatures for $n$ \cite{use1,CaCH10,CaGH10,CoLe13,GiHa12,Harris_2014}). To do so, methods have been proposed to identify these quantities from  far field data (for a range of frequencies). The first class of methods exploit the failure of sampling methods at these special frequencies to reconstruct them  \cite{CaCH10,Audithesis,AuHa14}. They require some a priori knowledge on the inclusion location and size.  Another class of methods exploit the inside-outside duality  revealed in \cite{DoSm92,EcPi95,DEPSU95,EcPi97} that relate internal resonant frequencies to spectral properties of the far field operator. These methods have been introduced  to identify TEs in \cite{KiLe13} and further developed in many subsequent works \cite{LaVa15,LeRe15,LePe15,LePe15bis}. This type of approaches only exploit the data and do not need a priori knowledge on the shape or location. However, the mathematical justification only applies to some simplified asymptotic configurations of the material properties  (small perturbation of large enough or small enough constants).\\
\newline
Besides, one of the main troubles with classical TEs is that their link with the index of refraction is not explicit nor easily accessible. Some monotonicity results have been obtained in \cite{CaGH10} but only for some of the TEs (see also  \cite{CaCoHa16} for other results related to transmission eigenvalues).  Recently, the idea of using an artificial background for which the associated TEs have a more direct connection with the material index of the inclusion has been introduced \cite{CCMM16,AuCH17,Cogar_2017,AuCH18}. \\
\newline
The goal of  this article is to develop the inside-outside duality for identifying TEs associated with artificial backgrounds. We investigate the cases where the modification of the background consists in introducing an artificial inhomogeneity with index $n_b$ that contains the targeted inclusion. A relative far field operator is defined as the difference between the far field operator from data and the far field operator associated with the background. Scaling this difference with the adjoint of the scattering matrix for the artificial background, one obtains an operator that has symmetric factorisation in terms of Hergltoz wave operators (see \cite{NaPT07,GMMR12,LaLe16,GrHa18}). The spectrum of this operator is then exploited to identify TEs. In the case when $n_b$ is constant, one is faced with the same difficulties as in \cite{KiLe13} ($n_b =1$) for the theoretical justification. We here exploit that the artificial background is a numerical tool that can be adapted at wish to propose the use of $n_b$ scaling like $\rho/k^2$ where $\rho$ is an arbitrary constant and where $k$ is the wavenumber. For this particular choice, one is capable of obtaining a necessary and sufficient condition for the identification of TEs from the behaviour of the phases of the eigenvalues of the modified far field operator. The main reason allowing for such result is that the space of generalized incident waves is independent from the wavenumber for the chosen scaling. In addition to this characterization, we also prove that one can recover the generalized incident wave. This result holds for any continuous dependence of $n_b$ with respect to $k$ and may be useful for the inverse spectral problem consisting in determining $n$ from the TEs.\\ 
\newline
The outline of this article is as follows. First, we present the notation and we define the classical TEs. Then we explain how, modifying the far field data by some quantity which can be computed numerically, we can derive a new interior transmission eigenvalue problem for which the identification of $n$ from the corresponding TEs is more simple. In Section \ref{SectionFactorisation}, we write a factorisation of the new far field operator which has been obtained in \cite[Section 3]{LaLe16}. In Section \ref{sectionFirstCharac}, we give a characterisation of the TEs via an intermediate operator appearing in the factorization. Then we describe the spectral behaviour of the new far field operator, first outside TEs (Section \ref{SectionOutsideTE}), then at TEs (Section \ref{SectionSpectralAtTEs}). Section \ref{SectionNum} is dedicated to the presentation of simple numerical experiments illustrating the theoretical results. Finally we sketch the proofs of some well-known technical results needed during the analysis in the Appendix. The two main results of this article are Theorems \ref{TheoPhase1} and \ref{TheoPhase2}.

\section{Classical transmission eigenvalues}
We consider the propagation of waves in a material with a localized penetrable inhomogeneity in time harmonic regime. The localized inhomogeneity is modelled by some bounded open set $\Om\subset\R^3$ with Lipschitz boundary $\partial\Om$ and a refractive index $n\in \mL^{\infty}(\R^3)$. We assume that $\R^3\setminus\overline{\Om}$ is connected, that $n$ is positive real valued and that $n=1$ in $\R^3\setminus\Om$. Let $u_i$ be a smooth function satisfying the Helmholtz equation $\Delta u_i+k^2 u_i=0$ in $\R^3$ at the wavenumber $k>0$. The scattering of $u_i$ by the inclusion is described by the problem
\begin{equation}\label{PbChampTotalFreeSpaceIntrod}
\begin{array}{|rcll}
\multicolumn{4}{|l}{\mbox{Find }u=u_i+u_s\mbox{ such that}}\\[2pt]
\Delta u+k^2n\,u & = & 0 & \mbox{ in }\R^3,\\[3pt]
\multicolumn{4}{|c}{\dsp{\dsp\lim_{r\to +\infty}\int_{|x|=r}\left|\frac{\partial u_s}{\partial r}-ik u_s \right|^2 ds(x)  = 0.}}
\end{array}
\end{equation}
The last line of (\ref{PbChampTotalFreeSpaceIntrod}) is the so-called Sommerfeld radiation condition. For all $k>0$, problem (\ref{PbChampTotalFreeSpaceIntrod}) has a unique solution  $u \in \mH^2_{\loc}(\R^3)$ \cite{CoKr13}. Moreover the scattered field $u_{s}$ has the expansion 
\begin{equation}\label{scatteredFieldFreeSpaceIntro}
u_{s}(r\theta_{s})= \dsp \cfrac{ e^{i k r}}{r}\,\Big(\,u_s^{\infty}(\theta_{s})+O(1/r)\,\Big),
\end{equation}
as $r=|x|\to+\infty$, uniformly in $\theta_{s}\in \mathbb{S}^2:=\{\theta\in\R^3;\;|\theta|=1\}$. The function $u_s^{\infty}: \mathbb{S}^2\to\Cplx$ is called the farfield pattern of $u_s$. Denote $B_R$ the ball centered at $O$ of radius $R>0$. The far field pattern has the representation
\begin{equation}\label{RepresentationFF}
u_s^{\infty}(\theta_{s})=\cfrac{1}{4\pi}\dsp\int_{\partial B_R}u_s\,\partial_{\nu}(e^{-ik\theta_s\cdot x})-\partial_{\nu} u_s\,\,e^{-ik\theta_s\cdot x}\,ds(x)
\end{equation}
for all $R>0$ such that $\supp(n-1)\subset B_R$. Here and elsewhere, $\nu$ stands for the outer unit normal to $B_R$. When $u_i$ coincides with the incident plane wave $u_{i}(\cdot,\theta_i):=e^{i k \theta_{i}\cdot x}$, of direction of propagation $\theta_{i}\in\mathbb{S}^2$, we denote respectively $u_s(\cdot,\theta_i)$, $u^{\infty}_s(\cdot,\theta_i)$ the corresponding scattered field and far field pattern.  
From the farfield pattern associated with one incident plane wave, by linearity we can define the farfield operator $F:\mL^2(\mathbb{S}^2)\to\mL^2(\mathbb{S}^2)$ such that 
\begin{equation}\label{defInitialFFoperator}
(Fg)(\theta_{s})=\int_{\mathbb{S}^2}g(\theta_{i})\,u_s^{\infty}(\theta_{s},\theta_{i})\,ds(\theta_{i}).
\end{equation}
The function $Fg$ corresponds to the farfield pattern of $u_s$ defined in (\ref{PbChampTotalFreeSpaceIntrod}) with 
\begin{equation}\label{HerglotzWave}
u_i=\int_{\mathbb{S}^2}g(\theta_{i})e^{ik\theta_{i}\cdot x}\,ds(\theta_{i}).
\end{equation}
We call Herglotz wave functions the incident fields of the above form. It is known, see \cite{CoKr01}, that the set $\{\int_{\mathbb{S}^2}g(\theta_{i})e^{ik\theta_{i}\cdot x}\,ds(\theta_{i})|_{\Om},\,g\in\mL^2(\mathbb{S}^2)\}$ 
is dense in $\{\varphi\in\mL^2(\Om);\,\Delta\varphi+k^2\varphi=0\mbox{ in }\Om\}$ for the norm of $\mL^2(\Om)$. From $F$, let us define the scattering operator $S:\mL^2(\mathbb{S}^2)\to\mL^2(\mathbb{S}^2)$ such that 
\begin{equation}\label{DefScaOp}
S:=\mrm{Id}+\frac{ik}{2\pi} F.
\end{equation}
It is known (see e.g. \cite[Theorem 4.4]{KiGr08}) that $S$ is unitary ($SS^{\ast}=S^{\ast}S=\mrm{Id}$) and that $F$ is normal ($FF^{\ast}=F^{\ast}F$).\\
\newline
In this article, we want to work with Transmission Eigenvalues (TEs). The latter are defined as the values of $k>0$ for which there exists a non zero generalized incident field which does not scatter. In other words, they are defined as the values of $k>0$ for which there exists $u_i\in\mL^2(\Om)\setminus\{0\}$ satisfying $\Delta u_i+k^2u_i=0$ in $\Om$ and such that the corresponding $u_s$ defined via (\ref{PbChampTotalFreeSpaceIntrod}) has a zero far field pattern. By the Rellich Lemma, then we have $u_s=0$ in $\R^3\setminus\overline{\Om}$. This leads to the definition of TEs as the values of $k\in\R^{\ast}_{+}:=(0;+\infty)$ for which the problem  
\begin{equation}\label{OriginalITEP}
\begin{array}{|rcll}
\Delta u_s+k^2 nu_s&=& k^2(1-n) u_i&\mbox{ in }\Om \\[2pt]
\Delta u_i+k^2u_i&=&0&\mbox{ in }\Om 
\end{array}
\end{equation}
admits non trivial solutions $(u_i,u_s)\in\mL^2(\Om)\times\mH^2_0(\Om)$.\\
\newline
TEs depend on the probe medium, in particular on the material index. In this work, we want to obtain quantitative, or at least qualitative, information on $n$ via the TEs by, roughly speaking, solving an inverse spectral problem. To proceed, one has to be able to identify TEs from the knowledge of the far field data. 
After the first approach introduced in \cite{CaCH10,Audithesis}, a more refined technique consists in using the inside-outside duality presented in 
\cite{KiLe13}. From the theoretical point of view however, the complete justification of this method is limited to the identification of the first eigenvalue and to particular indices $n$ which are constant and large or small enough. On the other hand, the determination of information on $n$ from the knowledge of TEs is not simple because the dependence of the eigenvalues in (\ref{OriginalITEP}) with respect to $n$ is not clear. In order to circumvent these two problems, following the idea introduced in \cite{CCMM16,AuCH17,AuCH18}, we will modify the data by some well-chosen quantity which can be computed numerically, so that the corresponding TEs offer a more direct information on $n$. Moreover, for these TEs, at least for certain artificial backgrounds below, we will obtain a general criterion of detection from far field data using the inside-outside duality.

\section{Transmission eigenvalues with an artificial background}
\renewcommand{\Om}{\Omega_b}

Let $n_b\in \mL^\infty(\R^3)$ be an arbitrary given real valued function such that $n_b=1$ in $\R^3\setminus\overline{\Om}$ where $\Om\supset \Omega$ is a Lipschitz domain with connected complement. This parameter has to be seen as the material index of an artificial background. Consider the scattering problem 
\begin{equation}\label{PbChampTotalFreeSpaceComp}
\begin{array}{|rcll}
\multicolumn{4}{|l}{\mbox{Find }u_b=u_i+u_{b,s}\mbox{ such that}}\\[2pt]
\Delta u_b+k^2n_bu_b & = & 0 & \mbox{ in }\R^3,\\[3pt]
\multicolumn{4}{|c}{\dsp{\dsp\lim_{r\to +\infty}\int_{|x|=r}\left|\frac{\partial u_{b,s}}{\partial r}-ik u_{b,s} \right|^2 ds(x)  = 0.}}
\end{array}
\end{equation}
When $u_i$ coincides with the incident plane wave $u_{i}(\cdot,\theta_i)=e^{i k \theta_{i}\cdot x}$, of direction of propagation $\theta_{i}\in\mathbb{S}^2$, we denote respectively $u_{b,s}(\cdot,\theta_{i})$, $u^{\infty}_{b,s}(\cdot,\theta_{i})$ the corresponding scattered field and far field pattern (see the expansion in (\ref{scatteredFieldFreeSpaceIntro})). From the farfield pattern, we define by linearity the farfield operator $F^b:\mL^2(\mathbb{S}^2)\to\mL^2(\mathbb{S}^2)$ such that 
\[
(F^bg)(\theta_{s})=\int_{\mathbb{S}^2}g(\theta_{i})\,u^{\infty}_{b,s}(\theta_{s},\theta_{i})\,ds(\theta_{i}).
\]
Similarly to (\ref{DefScaOp}), we introduce the scattering operator $S^b:\mL^2(\mathbb{S}^2)\to\mL^2(\mathbb{S}^2)$ such that
\begin{equation}\label{DefScaOpBkd}
S^b:=\mrm{Id}+\frac{ik}{2\pi} F^b.
\end{equation}
The operator $S^b$ is unitary while $F^b$ is normal. Finally, we define the relative farfield operator $F^{\rel}:\mL^2(\mathbb{S}^2)\to\mL^2(\mathbb{S}^2)$ as
\begin{equation}\label{DefArtFFop}
F^{\rel} :=F-F^{b}.
\end{equation}
In practice, $F$ is obtained from the measurements while $F^{b}$ has be to computed by numerically solving \eqref{PbChampTotalFreeSpaceComp} for a given $n_b$.   We remark that when $n_b\equiv1$, then $F^b\equiv0$. In this case, we simply have $F^{\rel} =F$.\\
\newline
The TEs for the artificial background are defined as the values of $k>0$ for which there exists $u_i\in\mL^2(\Om)\setminus\{0\}$ satisfying $\Delta u_i+k^2u_i=0$ in $\Om$ and such that the corresponding $u_s$, $u_{b,s}$ defined via (\ref{PbChampTotalFreeSpaceIntrod}), (\ref{PbChampTotalFreeSpaceComp}) are such that $u_s^{\infty}=u_{b,s}^{\infty}$. By the Rellich Lemma, this implies $u_s=u_{b,s}$ in $\R^3\setminus\overline{\Om}$. Thus $u=u_i+u_s$ must coincide with $u_b:=u_i+u_{b,s}$ in $\R^3\setminus\overline{\Om}$. Set $w:=(u-u_b)|_{\Om}\in\mH^2_0(\Om)$ and $v:=u_b{}|_{\Om}\in\mL^2(\Om)$. Computing the difference \eqref{PbChampTotalFreeSpaceIntrod}-\eqref{PbChampTotalFreeSpaceComp}, we find that $w$ satisfies $\Delta w+k^2nw=k^2(n_b-n)v$ in $\Om$. This leads us to define the TEs for the artificial background as the values of $k>0$ for which the problem  
\begin{equation}\label{NewITEP}
\begin{array}{|rcll}
\Delta w+k^2nw&=& k^2(n_b-n)v&\mbox{ in }\Om \\[2pt]
\Delta v+k^2n_bv&=&0&\mbox{ in }\Om 
\end{array}
\end{equation}
admits non trivial solutions $(v,w)\in\mL^2(\Om)\times\mH^2_0(\Om)$. Note that classical arguments (see e.g. \cite{CaHa13b}) allow one to show that if we have $n-n_b \ge \alpha >0$ in $\Om$ or $n_b- n \ge \alpha >0$ in $\Om$ for some constant $\alpha$, then the set of TEs of (\ref{NewITEP}) is discrete. Now, all the game consists in choosing $n_b$ cleverly so that the knowledge of the eigenvalues of (\ref{NewITEP}) give direct or interesting information on $n$. Remark that if one takes $n_b=0$ in $\Om$ (and $n_b=1$ in $\R^3\setminus\overline{\Om})$, one finds that \eqref{NewITEP} has a non zero solution if and only if the problem
\begin{equation}\label{NewITEPZIM}
\Delta(n^{-1}\Delta w) = -k^2 \Delta w \quad\mbox{ in }\Om 
\end{equation}
admits a non zero solution $w\in\mH^2_0(\Om)$. In \cite{AuCH18}, this situation is referred to as the Zero Index Material (ZIM) because $n_b=0$ in $\Om$. In opposition with problem (\ref{OriginalITEP}), one obtains a quite simple eigenvalue problem similar to the so-called plate buckling eigenvalue problem. Classical results concerning linear self-adjoint compact operators guarantee that the spectrum of (\ref{NewITEPZIM}) is made of real positive isolated eigenvalues of finite multiplicity $0<\lambda_1 \le \lambda_2 \le \dots \le \lambda_p \le\dots$ (the numbering is chosen so that each eigenvalue is repeated according to its multiplicity). Moreover, there holds $\lim_{p\to+\infty}\lambda_p=+\infty$ and we have the \textit{min-max} formulas
\begin{equation}\label{FormuleMinMax}
\lambda_p=\min_{E_p\in\mathscr{E}_p}\max_{w\in E_p\setminus\{0\}}\cfrac{(n^{-1}\Delta w,\Delta w)_{\mL^2(\Om)}}{\|\nabla w\|^2_{\mL^2(\Om)}}.
\end{equation}
Here $\mathscr{E}_p$ denotes the sets of subspaces $E_p$ of $\mH^2_0(\Om)$ of dimension $p$. Observe that the characterisation of the spectrum of problem (\ref{NewITEPZIM}) is much simpler than the one of problem (\ref{OriginalITEP}). Moreover, it holds under very general assumptions for $n$: we just require that $n|_{\Om}\in\mL^{\infty}(\Om)$ with $\mrm{ess}\inf_{\Om}n>0$.  In particular, $n$ can be equal to one inside $\Om$ (we do not need to know exactly the support $\overline{\Omega}$ of the defect in the reference medium, $\Om$ can be larger) and $n-1$ can change sign on the boundary. In comparison, the analysis of the spectrum of (\ref{OriginalITEP}) in such situations is much more complex and the functional framework must be adapted according to the values of $n$ (see \cite{Ches16} in the case where $n-1$ changes sign on $\partial\Om$). From (\ref{FormuleMinMax}), if we denote $\tilde{\lambda}_p$ the eigenvalues of (\ref{NewITEPZIM}) for another index $\tilde{n}\in\mL^{\infty}(\Om)$, we find that there holds, for all $p\in\N^{\ast}$, 
\[
\lambda_p\le \tilde{\lambda}_p\qquad\mbox{ when we have }\qquad n \ge \tilde{n}\qquad\mbox{a.e. in }\Om.
\]
The second advantage of considering problem (\ref{NewITEPZIM}) instead of problem (\ref{OriginalITEP}) is that the spectrum of (\ref{NewITEPZIM}) is entirely real. As a consequence, no information on $n$ is lost in complex eigenvalues which can exist a priori for the usual transmission problem \eqref{OriginalITEP} and which are hard to detect from the knowledge of $F$ for real frequencies. \\
\newline
Our goal in the following is to explain and to justify how to characterize the eigenvalues and the $v$ in (\ref{NewITEP}) from the knowledge of $F^{\rel}$ using the inside-outside duality approach. More specifically, we shall use the modified operator $\mathscr{F}$ defined as $\mathscr{F}:=(S^b)^{\ast}F^{\rel}$  for which it is well known that a symmetric factorization holds. We prove that this operator allows one to identify TEs of (\ref{NewITEP}) for the special choice of 
\begin{equation}
\label{choixdenb}
n_b = n_b(k) = \rho /k^2 \quad \mbox{ in } \Om,
\end{equation}
where $\rho \in \R$ is a constant independent from $k$. The particularity of this choice is that then the incident field $v$ in (\ref{NewITEP}) satisfies an equation independent from $k$ in $\Om$. As a consequence, a min-max criterion for TEs is derived on a Hilbert space independent from $k$ and a necessary and sufficient condition can be easily obtained.

\section{Factorisation of the modified farfield operator}\label{SectionFactorisation}

In order to develop the inside-outside duality, we first establish a (well known) symmetric factorisation of the operator $\mathscr{F}$ (see \cite{NaPT07,GMMR12,LaLe16,GrHa18}). We detail the procedure for the sake of clarity, following the approach in \cite{LaLe16}.\\
\newline
\textbf{Step 1.} For $g\in\mL^2(\mathbb{S}^2)$, define the function $v_g$ such that 
\begin{equation}\label{vartg}
v_g (x) :=\int_{\mathbb{S}^2}g(\theta_{i}){u}_b(x,\theta_{i}) \,ds(\theta_{i})
\end{equation}
where ${u}_b(\cdot,\theta_{i})$ appears after (\ref{PbChampTotalFreeSpaceComp}). Then define the operator $H: \mL^2(\mathbb{S}^2) \to \mL^2(\Om)$ such that 
\begin{equation}\label{defHart}
H g = v_g|_{\Om}.
\end{equation}
From the definitions of $F$ and $F^b$ we observe that $F^{\rel}g=w^{\infty}$ where $w^{\infty}$ is the far field pattern associated with the solution of the problem
\begin{equation}\label{eqwart}
\left|\begin{array}{l}
\mbox{Find }w \in \mH^2_{\loc}(\R^3)\mbox{ such that}\\[2pt]
\Delta w+k^2n\,w  = k^2(n_b-n) v   \mbox{ in }\R^3,\\[5pt]
{\dsp\lim_{r\to +\infty}\int_{|x|=r}\left|\frac{\partial w}{\partial r}-ik w \right|^2 ds(x)  = 0,}
\end{array}
\right.
\end{equation}
with $v=v_g$. More generally, define the operator $G:\mL^2(\Om)\to\mL^2(\mathbb{S}^2)$ such that $Gv=w^{\infty}$, $w^{\infty}$ being the far field pattern of the function $w$ solving (\ref{eqwart}). With such notation, we can factorize $F^{\rel}$ as 
\begin{equation}\label{factoFGHart}
F^{\rel}=G\herg.
\end{equation}
The properties of problem (\ref{eqwart}) will play a key role in the analysis below. Before proceeding, we state a classical result whose proof is sketched in the Appendix.
\begin{proposition}\label{PropoEstimRegu}
For all $k\in\R^{\ast}_{+}$, the solution of (\ref{eqwart}) satisfies the estimate
\begin{equation}\label{EstimRegu}
\|w\|_{\mH^2(\Om)} \le C\,\|v\|_{\mL^2(\Om)},
\end{equation}
where the constant $C$ is independent from $v\in\mL^2(\Om)$. Moreover, for any compact set $I\subset\R^{\ast}_{+}$, the constant in (\ref{EstimRegu}) can be chosen independent from $k\in I$. 
\end{proposition}
\noindent \textbf{Step 2.} We express the adjoint of $\herg$ in terms of far fields. To proceed, for a given $f\in\mL^2(\Om)$ (extended by zero outside of $\Om$), we work with $\psi\in\mH^2_{\loc}(\R^3)$ solving the problem
\[
\Delta\psi +k^2 n_b \psi =-f\quad\mbox{ in }\R^3
\]
together with the Sommerfeld radiation condition. Integrating by parts in $B_R$ for $R>0$ large enough, using that $\Delta v_g+k^2n_bv_g=0$ in $\R^3$ and replacing $v_g$ by its expression given in (\ref{vartg}), we find, where $(\cdot, \cdot)_{\Om}$ denotes the $\mL^2(\Om)$ scalar product,  
\begin{equation}\label{IPPMul}
\begin{array}{lcl}
(f,\herg g)_{\Om} & = & -\dsp\int_{B_R}(\Delta\psi+k^2n_b\psi)\,\overline{v_g}\,dx \\[9pt]
&=&  \dsp\int_{B_R}\nabla\psi\cdot\nabla\overline{v_g}-k^2n_b\psi\,\overline{v_g}\,dx -\dsp\int_{\partial B_R}\partial_{\nu} \psi\,\,\overline{v_g}\,ds(x)\\[9pt]
&=& \dsp\int_{\partial B_R}\psi\,\partial_{\nu}\int_{\mathbb{S}^2}\overline{g(\theta_{i})}(e^{-ik\theta_i\cdot x}+\overline{u_{b,s}(x,\theta_{i})}) \,ds(\theta_{i})\,ds(x)\\[10pt]
&& \hspace{0cm}-\dsp\int_{\partial B_R}\partial_{\nu} \psi\,\,\dsp\int_{\mathbb{S}^2}\overline{g(\theta_{i})}(e^{-ik\theta_i\cdot x}+\overline{u_{b,s}(x,\theta_{i})}) \,ds(\theta_{i})\,ds(x).
\end{array}
\end{equation}
On the other hand, formula (\ref{RepresentationFF}) gives, for $\theta_s\in\mathbb{S}^2$, 
\[
\psi^{\infty}(\theta_s)=\cfrac{1}{4\pi}\dsp\int_{\partial B_R}\psi\,\partial_{\nu}(e^{-ik\theta_s\cdot x})-\partial_{\nu} \psi\,\,e^{-ik\theta_s\cdot x}\,ds(x),
\]
when $R>0$ is such that $\supp(n_b-1)\subset B_R$. From (\ref{IPPMul}), taking the limit as $R\to+\infty$ and using the radiation condition to deal with the term involving $u_{b,s}$, we get 
\begin{equation}\label{defHstar}
\begin{array}{lcl}
(f,\herg g)_{\Om}&=&4\pi(\psi^{\infty},g)_{\mathbb{S}^2}+2ik(\psi^{\infty},F^bg)_{\mathbb{S}^2}\\[4pt]
&=&4\pi((\mrm{Id}-\cfrac{ik}{2\pi}(F^b)^{\ast})\psi^{\infty},g)_{\mathbb{S}^2} \ = \ 4\pi((S^b)^{\ast}\psi^{\infty},g)_{\mathbb{S}^2}.
\end{array}
\end{equation}
This implies $\herg^{\ast}f=4\pi(S^b)^{\ast}\psi^{\infty}$ and so $S^b\herg^{\ast}f=4\pi \psi^{\infty}$.\\
\newline
\textbf{Step 3.} Observing that the first line of (\ref{eqwart}) can also be written as 
\begin{equation}\label{Rewriting}
\Delta w+k^2n_b\,w  = -f\qquad\mbox{ with }\quad f=k^2(n-n_b) (v_g+w),
\end{equation}
we define the operator  $T:\mL^2(\Om) \to \mL^2(\Om)$ such that 
\begin{equation}\label{defTtC}
T v = \frac{1}{4\pi}  k^2(n - n_b) (v + w),
\end{equation}
where $w \in \mH^2_{\loc}(\R^3)$ is the unique solution of (\ref{eqwart}). Then clearly we have $G=S^b\herg^{\ast}T$ which implies the factorisation
\[
F^{\rel}=S^b\herg^{\ast}T\herg.
\]
\textbf{Step 4.} Finally, we define the operators $\mathscr{F}:\mL^2(\mathbb{S}^2) \to \mL^2(\mathbb{S}^2)$ and $\mathscr{S}:\mL^2(\mathbb{S}^2) \to \mL^2(\mathbb{S}^2)$ such that 
\begin{equation}\label{ModifiedOp}
\mathscr{F}:=(S^b)^{\ast}F^{\rel}\qquad\mbox{ and }\qquad \mathscr{S}:=\mrm{Id}+\frac{ik}{2\pi}\mathscr{F}.
\end{equation}
We have $\mathscr{F}=(S^b)^{\ast}(F-F^b)=2\pi(ik)^{-1}(S^b)^{\ast}(S-S^b)=2\pi(ik)^{-1}((S^b)^{\ast}S-\mrm{Id})$ and so
\[
\mathscr{S}=(S^b)^{\ast}S.
\]
Since $S$ and $S^b$ are unitary, we deduce that $\mathscr{S}$ is unitary (and so normal). And one can verify that this is enough to guarantee that $\mathscr{F}$ is normal. We gather these results in the following statement.
\begin{proposition}\label{PropFunitary}
$i)$ The operator $\mathscr{F}:\mL^2(\mathbb{S}^2) \to \mL^2(\mathbb{S}^2)$ admits the factorisation
\begin{equation}\label{ModifiedFacto}
\mathscr{F}=\herg^{\ast}T\herg,
\end{equation}
where $\herg:\mL^2(\mathbb{S}^2) \to \mL^2(\Om)$ and $T:\mL^2(\Om) \to \mL^2(\Om)$ are respectively defined in (\ref{defHart}) and (\ref{defTtC}).\\[4pt]
$ii)$ The operator $\mathscr{F}$ is normal and $\mathscr{S}$ is unitary.
\end{proposition}
\begin{remark}
When we take $n_b=1$ in $\R^3$, we have $F^b=0$ and so $S^b=\mrm{Id}$. In this case, (\ref{ModifiedFacto}) is nothing but the classical factorisation of the operator $F$ (see e.g. \cite{KiGr08}). 
\end{remark}
\noindent In the next proposition, we detail some properties of the operators appearing in the factorization (\ref{ModifiedFacto}) which will be useful in the analysis below.
\begin{proposition}\label{propoCompact}
$i)$ $\herg:\mL^2(\mathbb{S}^2)\to\mL^2(\Om)$ is compact injective and $\mathscr{F}:\mL^2(\mathbb{S}^2)\to\mL^2(\mathbb{S}^2)$ is compact.\\[4pt]
$ii)$ The operator $T$ satisfies the energy identity
\begin{equation}\label{identity1}
\forall v\in\mL^2(\Om), \qquad 4\pi\,\Im m\, (T v, v)_{\Om} =  k \int_{\mathbb{S}^{2}} |w^{\infty}|^2\,ds,
\end{equation}
where $w^{\infty}$ is the far field pattern of the solution of (\ref{eqwart}).\\[4pt]
$iii)$ Assume that we have $n-n_b \ge \alpha >0$ in $\Om$ or $n_b- n \ge \alpha >0$ in $\Om$ for some constant $\alpha$. Then $T$ admits the decomposition 
\begin{equation}\label{FredholmDecompo}
T=A+K,
\end{equation}
where $K:\mL^2(\Om)\to\mL^2(\Om)$ is compact and where $A:\mL^2(\Om)\to\mL^2(\Om)$ is an isomorphism such that 
\[
\exists c>0,\,\forall v\in\mL^2(\Om),\qquad|(Av,v)_{\Om}| \ge c\,\|v\|^2_{\mL^2(\Om)}\qquad\mbox{($A$ is coercive)}.
\]
\end{proposition}
\begin{proof}
$i)$ We have $\herg g=(u_i+u_{b,s})|_{\Om}$ where $u_{b,s}$ is the scattered field of the solution of (\ref{PbChampTotalFreeSpaceComp}) with $u_i=\int_{\mathbb{S}^2}g(\theta_{i})e^{ik\theta_{i}\cdot x}\,ds(\theta_{i})|_{\Om}$. Since $g\mapsto \int_{\mathbb{S}^2}g(\theta_{i})e^{ik\theta_{i}\cdot x}\,ds(\theta_{i})|_{\Om}$ is compact from $\mL^2(\mathbb{S}^2)$ to $\mL^2(\Om)$ and since $u_i\mapsto u_{b,s}$ is continuous from $\mL^2(\Om)\to\mH^2(\Om)$ (similar to Proposition \ref{PropoEstimRegu}), we deduce that $\herg:\mL^2(\mathbb{S}^2)\to\mL^2(\Om)$ is compact. Injectivity of $\herg$ can be proved as for the classical Herglotz operator (see e.g. \cite{CoKr13}). Finally, since $\mathscr{F}=\herg^{\ast}T\herg$ and since $T:\mL^2(\Om)\to\mL^2(\Om)$ is continuous (consequence of Proposition \ref{EstimRegu}), we deduce that $\mathscr{F}:\mL^2(\mathbb{S}^2)\to\mL^2(\mathbb{S}^2)$ is compact.\\[4pt]
$ii)-iii)$ From the formulation (\ref{eqwart}), we see that for all $v$, $v'\in\mL^2(\Om)$, we have 
\begin{equation}\label{FirstEstim}
4\pi\,(T v,v')_{\Om}= k^2\dsp\int_{\Om}(n-n_b) (v + w)\overline{v'}\,dx = k^2\dsp\int_{\Om}(n-n_b)v\overline{v'}\,dx+k^2\dsp\int_{\Om}(n-n_b) w\overline{v'}\,dx.
\end{equation}
With the Riesz representation theorem, introduce the linear bounded operators $A,\,K:\mL^2(\Om)\to\mL^2(\Om)$ such that for all $v$, $v'\in\mL^2(\Om)$,
\[
(Av,v')_{\Om}=\cfrac{k^2}{4\pi}\dsp\int_{\Om}(n-n_b)v\overline{v'}\,dx,\qquad(Kv,v')_{\Om}=\cfrac{k^2}{4\pi}\dsp\int_{\Om}(n-n_b)w\overline{v'}\,dx.
\]
It is clear that $A$ is an isomorphism when we have $n-n_b \ge \alpha >0$ in $\Om$ or $n_b- n \ge \alpha >0$  in $\Om$. On the other hand, using that the map $v\mapsto w$ from $\mL^2(\Om)$ to $\mH^2(D)$ is continuous for any bounded domain $D\subset\R^3$ and that $\mH^2(\Om)$ is compactly embedded in $\mL^2(\Om)$, one can show that $K$ is compact. This proves item $iii)$.  Then taking $v'=v$ in (\ref{FirstEstim}), since $n-n_b$ is real valued, we get
\[
\begin{array}{lcl}
4\pi\,(T v, v)_{\Om} & = & k^2\dsp\int_{\Om}(n-n_b)|v|^2\,dx+k^2\dsp\int_{\Om}(n-n_b) w\overline{v}\,dx\\[10pt]
& = & k^2\dsp\int_{\Om}(n-n_b)|v|^2\,dx-\dsp\int_{B_R}w(\Delta\overline{w}+k^2n\overline{w})\,dx,
\end{array}
\]
for $R>0$ large enough. Integrating by parts in $B_R$ and using the radiation condition, one gets 
\[
4\pi\,(T v, v)_{\Om}=k^2\dsp\int_{\Om}(n-n_b)|v|^2\,dx+\dsp\int_{B_R}
|\nabla w|^2-k^2n|w|^2\,dx+ik\dsp\int_{\partial B_R}|w|^2\,ds+O(1/R).
\]
Taking the imaginary part and the limit as $R\to+\infty$, we obtain the important identity (\ref{identity1}) which proves item $ii)$. 
\end{proof}
\begin{remark}
We observe that working with the scattering operator $\mathscr{S}=(S^b)^{\ast}S$ consists, roughly speaking, in proceeding to the following series of experiments. The observer sends some Herglotz incident wave with density $g$ in the probed medium characterized by the index $n$, collects the far field measurements and then backpropagates the Herglotz wave with density $Sg$ through the artificial background characterized by the index $n_b$. If $k_0$ is a TE of (\ref{NewITEP}), then for any (small) $\eps$, there is a density $g\in\mL^2(\mathbb{S}^2)$ such that $\|g-\mathscr{S}g\|_{\mL^2(\mathbb{S}^2)} \le \eps\,\|g\|_{\mL^2(\mathbb{S}^2)}$ (the input signal is almost the same as the output signal). Note that if $n_b=n$, for any incident field, the observer gets at the end of the experiments the original signal. This is coherent with the fact that when $n_b=n$, there holds $S^b=S$ and so $\mathscr{S}=\mrm{Id}$.  
\end{remark}

\section{Characterisation of transmission eigenvalues via $T$}\label{sectionFirstCharac}
Our final goal is to identify TEs and transmission eigenfunctions (at least the $v$ in (\ref{NewITEP})) from the knowledge of $\mathscr{F}$, which itself can be computed from $F$ (the data). To proceed, as an intermediate step, following the presentation of \cite{KiLe13}, we explain in this section the connection existing between the TEs of (\ref{NewITEP}) (corresponding to an artificial background) and the properties of the operator $T$ appearing in the factorisation $\mathscr{F}=\herg^{\ast}T\herg$. To proceed, we denote by $\mrm{H}_{\mrm{inc}}$ the closure of the range of $H$ in $\mL^2(\Om)$.  
\begin{proposition}
We have  $\mrm{H}_{\mrm{inc}} =\{\varphi \in \mL^2(\Om); \; \Delta\varphi+k^2n_b \varphi = 0 \mbox{ in } \Om\}$.
\end{proposition}
\begin{proof}
Clearly $H(\mL^2(\mathbb{S}^2))\subset\mrm{H}_{\mrm{inc}}$ and $\mrm{H}_{\mrm{inc}}$ is a closed subset of $\mL^2(\Om)$. Therefore the closure of the range of $H$ in $\mL^2(\Om)$ is a subset of $\mrm{H}_{\mrm{inc}}$. Now consider some $f\in\mrm{H}_{\mrm{inc}}$ such that 
\[
(Hg,f)_{\Om}=(g,H^{\ast}f)_{\mathbb{S}^2}=0,\qquad \forall g\in\mL^2(\mathbb{S}^2).
\]
This is equivalent to have $H^{\ast}f=0$. From (\ref{defHstar}), we have $\herg^{\ast}f=4\pi(S^b)^{\ast}\psi^{\infty}$ where $\psi^{\infty}$ is the far field pattern of the outgoing function $\psi\in\mH^2_{\loc}(\R^3)$ satisfying $\Delta\psi +k^2 n_b \psi =-f$ in $\R^3$. Since $(S^b)^{\ast}$ is an isomorphism of $\mL^2(\mathbb{S}^2)$, we deduce $\psi^{\infty}=0$. 
From the Rellich lemma, we infer that $\psi=0$ in $\R^3\setminus\overline{\Om}$. Then integrating by parts in $\Om$, we get
\[
-\|f\|_{\mL^2(\Om)}^2=\int_{\Om}(\Delta\psi +k^2 n_b \psi)\overline{f}\,dx=\int_{\Om}\psi(\Delta \overline{f} +k^2 n_b \overline{f})\,dx=0
\]
(because $f\in\mrm{H}_{\mrm{inc}}$). Thus $f=0$ which shows that the desired result.
\end{proof}
\noindent From the energy identity (\ref{identity1}), we see that we have 
\[
\Im m\,(Tv,v)_{\Om}\ge 0,\qquad \forall v\in\mrm{H}_{\mrm{inc}}.
\]
The following proposition ensures that the TEs of (\ref{NewITEP}) coincide exactly with the frequencies $k>0$ such that the form $v\mapsto \Im m\,(Tv,v)_{\Om}$ is not definite positive in $\mrm{H}_{\mrm{inc}}$.
\begin{proposition}\label{propoTET}
Assume that we have $n-n_b \ge \alpha >0$ in $\Om$ or $n_b- n \ge \alpha >0$ in $\Om$ for some constant $\alpha$. Then
\[
\mbox{$k$ is a TE of (\ref{NewITEP})}\qquad\Leftrightarrow\qquad \exists v\in\mrm{H}_{\mrm{inc}}\setminus\{0\}\mbox{ such that }\Im m\,(Tv,v)_{\Om}=0.
\]
Moreover if $k$ is a TE of (\ref{NewITEP}), then $(v,w|_{\Om})\in\mL^2(\Om)\times\mH^2_0(\Om)$, where $w$ is defined as the solution of (\ref{eqwart}), is a corresponding eigenpair. And we have $(Tv,v)_{\Om}= 0$.
\end{proposition}
\begin{proof}
Assume that $k$ is a TE of (\ref{NewITEP}). Let $(v,w)\in\mL^2(\Om)\times\mH^2_0(\Om)$ be an associated non zero eigenpair. Then obviously $v\in\mrm{H}_{\mrm{inc}}\setminus\{0\}$ ($v$ is not null otherwise we would have $w\equiv0$). Using (\ref{Rewriting}) and integrating twice by parts, then we find 
\begin{equation}\label{IPPT}
\begin{array}{lcl}
(Tv,v)_{\Om}=\dsp\frac{1}{4\pi}\int_{\Om}k^2(n - n_b) (v + w) \overline{v}\,dx&=&-\dsp\frac{1}{4\pi}\int_{\Om} (\Delta w+k^2n_bw) \overline{v}\,dx\\[10pt]
&=&-\dsp\frac{1}{4\pi}\int_{\Om} w (\Delta \overline{v}+k^2n_b\overline{v})\,dx\ =\ 0.
\end{array}
\end{equation}
Conversely, assume that $v\in\mrm{H}_{\mrm{inc}}\setminus\{0\}$ is such that $\Im m\,(Tv,v)_{\Om}=0$. Then from (\ref{identity1}), we have $w^{\infty}=0$. From the Rellich lemma, this implies $w=0$ in $\R^3\setminus\overline{\Om}$. Then $(v,w|_{\Om})$ belongs to $\mL^2(\Om)\times\mH^2_0(\Om)$. This shows that $k$ is a TE of (\ref{NewITEP}). In this case, from (\ref{IPPT}), we infer that $(Tv,v)_{\Om}= 0$.
\end{proof}

\section{Spectral properties of $\mathscr{F}$ outside transmission eigenvalues}\label{SectionOutsideTE}

\begin{figure}[!ht]
\centering
\begin{tikzpicture}[scale=1]
\draw (0,2) circle (2);
\draw[->] (0,-0.2)--(0,4.4);
\draw[->] (-2.2,0)--(2.2,0);
\draw (-0.1,2)--(0.1,2);
\node at(0.7,2){$2\pi i/k$};
\node at(-1.65,3.75){$\lambda_2$};
\node at(2.4,2.75){$\lambda_3$};
\node at(-2.4,1.8){$\lambda_\star$};
\node at(-0.2,-0.2){$0$};
\foreach \x in {0.3,0.4,0.5,1,2,...,50}
\pgfmathtruncatemacro{\tmp}{110/\x}
\draw[fill=blue,draw=none] ({2*cos(\tmp-90)},{2+2*sin(\tmp-90)}) circle (2.2pt);
\end{tikzpicture}\qquad\qquad 
\raisebox{0.3cm}{\begin{tikzpicture}[scale=1]
\draw (0,0) circle (2);
\draw[->] (0,-2.2)--(0,2.4);
\draw[->] (-2.2,0)--(2.2,0);
\draw (-0.1,2)--(0.1,2);
\node at(-0.5,1.5){$e^{i\delta_3}$};
\node at(-1.1,-1){$e^{i\delta_2}$};
\node at(0.45,-1.65){$e^{i\delta_{\star}}$};
\node at(-0.2,-0.2){$0$};
\node at(2.1,-0.2){$1$};
\foreach \x in {0.3,0.4,0.5,1,2,...,50}
\pgfmathtruncatemacro{\tmp}{110/\x}
\draw[fill=blue,draw=none] ({2*cos(\tmp)},{2*sin(\tmp)}) circle (2.2pt);
\end{tikzpicture}}
\caption{Left: eigenvalues of $\mathscr{F}$. Right: eigenvalues of $\mathscr{S}$. Here the representation corresponds to a situation where $n-n_b\ge\alpha>0$ in $\Om$.\label{Illustration}}
\end{figure}
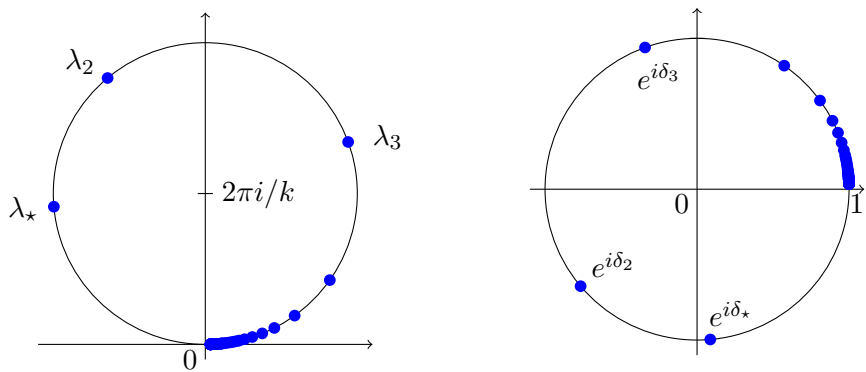

\noindent In Proposition \ref{propoTET}, we have established a clear connection between the TEs of problem (\ref{NewITEP}) and the kernel of the positive form $v\mapsto \Im m\,(Tv,v)_{\Om}$ in $\mrm{H}_{\mrm{inc}}$. We would like to translate this into a criterion on the kernel of the form $g\mapsto \Im m\,(\mathscr{F}g,g)_{\mathbb{S}^2}$ in $\mL^2(\mathbb{S}^2)$ using the factorisation $\mathscr{F}=H^{\ast}TH$ which implies the relation $(\mathscr{F}g,g)_{\mathbb{S}^2}=(T\herg g,\herg g)_{\Om}$ for all $g\in\mL^2(\mathbb{S}^2)$. However, this is not that simple due to the fact that the range of $\herg$ is not closed in $\mL^2(\Om)$. Said differently, if $v\in\mrm{H}_{\mrm{inc}}$ is such that $\Im m\,(Tv,v)_{\Om}=0$, in general there is no $g\in\mL^2(\mathbb{S}^2)$ such that $v=\herg g$ (for this particular point, see the articles \cite{BlPS14,PaSV14,ElHu15,ElHu18}). Instead, in the rest of the article we will study the way the eigenvalues of $\mathscr{F}$ accumulate at zero. First, we consider the situation when $k$ is not a transmission eigenvalue. \\
\newline 
Since $\mathscr{F}:\mL^2(\mathbb{S}^2)\to\mL^2(\mathbb{S}^2)$ is normal and compact (Propositions \ref{PropFunitary} and \ref{propoCompact}), there is an orthonormal complete basis $(g_j)_{j=1}^{+\infty}$ of $\mL^2(\mathbb{S}^{2})$ such that 
\begin{equation}\label{spectrum}
\mathscr{F}g_j = \lambda_j g_j,
\end{equation}
where the $\lambda_j$ are the eigenvalues of $\mathscr{F}$. The terms of the sequence $(\lambda_j)$ are complex numbers that accumulate at zero. Moreover, since $\mathscr{S}$ is unitary, the eigenvalues of $\mathscr{F}$ lie on the circle of radius $2\pi/k$ and center $2\pi i/k$ (see Figure \ref{Illustration} left).\\
\newline
Assume that we have $n-n_b \ge \alpha >0$ in $\Om$ or $n_b-n \ge \alpha >0$ in $\Om$ for some constant $\alpha$. In this case, if $k$ is not a transmission eigenvalue then $\mathscr{F}$ is injective and so the $\lambda_j$ in (\ref{spectrum}) are all non zero. Indeed, if $g\in\ker\,\mathscr{F}$, we must have $0=(\mathscr{F}g,g)_{\mathbb{S}^2}=(T\herg g,\herg g)_{\Om}$. From Proposition \ref{propoTET}, we deduce $\herg g=0$ and so $g=0$ because $H$ is injective (Proposition \ref{propoCompact} $i)$). In this situation, we denote by 
\begin{equation}\label{DefDelta}
e^{i\delta_j},\qquad\mbox{ with }\delta_j\in(0;2\pi),
\end{equation}
the eigenvalues of $\mathscr{S}$ (see Figure \ref{Illustration} right). With this notation, we have
\[
\lambda_j=\cfrac{2\pi}{ik}\,(e^{i\delta_j}-1)\qquad\Leftrightarrow\qquad e^{i\delta_j}=1+\cfrac{ik}{2\pi}\,\lambda_j.
\]
The only possible accumulation points for the sequence $(\delta_j)$ are $0$ and $2\pi$. Using the decomposition $T=A+K$ obtained in (\ref{FredholmDecompo}) where the sign of the isomorphism $A$ is known and where $K$ is compact, one can get the following proposition. Its proof is exactly the same as the one of \cite[Lemma 4.1]{KiLe13} (see also \cite[Theorem 2.25]{CaCoHa16}). 
\begin{proposition}\label{PropoOutsideTEs}
Assume that $k$ is not a TE of the problem (\ref{NewITEP}). \\[3pt]
- If $n-n_b\ge\alpha>0$ in $\Om$, then the sequence $(\delta_j)$ accumulates only at $0$ (see Figure \ref{Illustration} right).\\[3pt]
- If $n_b-n\ge\alpha>0$ in $\Om$, then the sequence $(\delta_j)$ accumulates only at $2\pi$. 
\end{proposition}
\noindent When $k$ is not a TE of the problem (\ref{NewITEP}), from Proposition \ref{PropoOutsideTEs} we deduce that we can order the phases of the eigenvalues of $\mathscr{S}$ so that
\[
\begin{array}{ll}
2\pi>\delta_1 \ge \delta_2 \ge \cdots \ge \delta_j\ge\cdots>0 &\quad\mbox{ when }n-n_b\ge\alpha>0\\[4pt]
0<\delta_1 \le \delta_2 \le \cdots \le \delta_j\le\cdots<2\pi &\quad\mbox{ when }n_b-n\ge\alpha>0.
\end{array}
\]
In the analysis below, the quantity $\delta_1$ will play a particular role. We set 
\[
\delta_{\star}=\delta_1\qquad\mbox{ and }\qquad \lambda_{\star}=\cfrac{2\pi}{ik}\,(e^{i\delta_{\star}}-1).
\]

\section{Spectral properties of $\mathscr{F}$ at transmission eigenvalues}\label{SectionSpectralAtTEs}

Now we study, the behaviour of the eigenvalues of $\mathscr{F}$ as $k$ tends to a TE of the problem (\ref{NewITEP}). In the analysis below, we will conduct calculus with varying $k$. As a consequence, we will explicitly indicate the dependence on $k$, denoting for example the far field operator by $\mathscr{F}(k)$. Moreover, in order to include the case of $n_b$ depending on $k$ as in \eqref{choixdenb}, we make the additional assumption that the mapping $k \mapsto n_b(k)$ is continuously differentiable from $\R_+^*$ to $\mL^\infty(\Om)$.\\
\newline
We start with a technical result whose (classical) proof is given in the Appendix. If $\mX$, $\mY$ are two Banach spaces, we denote $\mathcal{L}(X,Y)$ the set of linear bounded operators from $\mX$ to $\mY$. This space is endowed with the usual operator norm
\[
\|L\|:=\sup_{\varphi\in\mX\setminus\{0\}}\cfrac{\|L\varphi\|_{\mY}}{\|\varphi\|_{\mX}},\qquad\forall L\in\mathcal{L}(X,Y).
\]
\begin{proposition}\label{ContinuousDependancek}~\\
The mapping $k \mapsto \herg(k)$ is continuous from $\R^{\ast}_{+}$ to $\mathcal{L}(\mL^2(\mathbb{S}^2),\mL^2(\Om))$.\\
The mapping $k \mapsto T(k)$ is continuous from $\R^{\ast}_{+}$ to $\mathcal{L}(\mL^2(\Om),\mL^2(\Om))$.\\
The mapping $k \mapsto \mathscr{F}(k)$ is continuous from $\R^{\ast}_{+}$ to $\mathcal{L}(\mL^2(\mathbb{S}^2),\mL^2(\mathbb{S}^2))$.
\end{proposition}

\subsection{A sufficient condition for the detection of transmission eigenvalues}
In this paragraph, we provide a sufficient condition allowing one to detect TEs of (\ref{NewITEP}). This is the first main result of the article. The first part of the statement is similar to \cite{EcPi95,KiLe13} (see also \cite[Theorem 4.47]{CaCoHa16}). The second part concerning the identification of $v$ where $(v,w)$ is an eigenpair of (\ref{NewITEP}) is new.
\begin{theorem} \label{TheoPhase1}
Let $k_0>0$ and $I=(k_0-\eps;k_0-\eps)\setminus\{k_0\}$ such that no $k\in I$ is a TE of (\ref{NewITEP}). Assume that there is a sequence $(k_j)$ of elements of $I$ such that 
\[
\lim_{j\to+\infty} k_j=k_0\qquad\mbox{ and }\qquad \lim_{j\to+\infty}\delta_{\star}(k_j)=
\begin{array}{|ll}
2\pi & \mbox{ when }n-n_b(k_0)\ge\alpha>0\\[4pt]
0 & \mbox{ when }n_b(k_0)-n\ge\alpha>0.
\end{array}
\]
Then $k_0$ is a TE of (\ref{NewITEP}). Moreover, the sequence $(v_j)$, with
$$
v_j := \frac{\herg(k_j) g_{j}}{{\|\herg(k_j) g_{j} \|_{\mL^2(\Om)}}}, 
$$
admits a subsequence which converges strongly to $v\in\mL^2(\Om)$, where $(v,w)$ is an eigenpair of (\ref{NewITEP}) associated with $k_0$. Here $g_{j}$ is a normalised eigenfunction of $\mathscr{F}$ associated with $\lambda_{\star}(k_j)$ and $w$ is the solution of \eqref{eqwart} with $k=k_0$.
\end{theorem}
\begin{remark}
The additional information on eigenfunctions given in Theorem \ref{TheoPhase1}  can be helpful in solving the inverse spectral problem of determining $n$ from TEs of (\ref{NewITEP}).
\end{remark}
\begin{proof}
Let us prove the first part of the statement. We consider only the case $ n-n_b(k_0) \ge \alpha >0$.  The case $n_b(k_0)- n \ge \alpha >0$ follows from the same arguments replacing $\mathscr{F}(k)$ with $-\mathscr{F}(k)$. Considering $j$ sufficiently large we can assume  that 
$$
 n-n_b(k_j) \ge \alpha/2 >0.
$$
Set
$$
\psi_j := \frac{\herg(k_j) g_{j}}{\sqrt{|\lambda_{\star}(k_j)|}} \in\mL^2(\Om).
$$
The sequence $(\psi_j)$ satisfies, according to the assumptions and the factorisation \eqref{ModifiedFacto}, 
\begin{equation} \label{fedi1}
({T(k_j) \psi_j},{ \psi_j})_{\Om} = \cfrac{\lambda_{\star}(k_j)}{|\lambda_{\star}(k_j)|}\,(g_{j}
, g_{j})_{\mathbb{S}^{2}}= \cfrac{\lambda_{\star}(k_j)}{|\lambda_{\star}(k_j)|} \underset{j\to+\infty}{\to} -1.
\end{equation}
Using Proposition \ref{propoCompact} and working by contradiction, one can verify that if $k_0$ is not a TE of (\ref{NewITEP}), then $T(k_0)$ is coercive in $\mrm{H}_{\mrm{inc}}(k_0)$. Introduce $P(k):\mL^2(\Om)\to\mrm{H}_{\mrm{inc}}(k)$ the projection in $\mrm{H}_{\mrm{inc}}(k)$ for the inner product of $\mL^2(\Om)$. Using that the maps $k \mapsto P(k)$ and $k \mapsto T(k)$ (Proposition \ref{ContinuousDependancek}) are continuous in the operator norm, we deduce that the $T(k_j)$ are uniformly coercive in $\mrm{H}_{\mrm{inc}}(k)$ for $j$ sufficiently large. Identity \eqref{fedi1} then shows that the sequence $(\psi_j)$ is bounded in $\mL^2(\Om)$ and consequently,  up to a subsequence, one can assume that $(\psi_j)$ weakly converges to some $\psi_0$ in $\mL^2(\Om)$. Since $\psi_j\in\mrm{H}_{\mrm{inc}}(k_j)$ for all $j\in\N$, the weak limit satisfies in the sense of distributions 
$$
\Delta \psi_0 + k_0^2n_b(k_0)\psi_0 =0 \mbox{ in } \Om, 
$$
meaning that  $\psi_0 \in \mrm{H}_{\mrm{inc}}(k_0)$. Let us denote by
$w_j \in\mH^2_\mathrm{loc}(\R^3)$ (resp. $w_0 \in\mH^2_\mathrm{loc}(\R^3)$) the solution of \eqref{eqwart} with $v=\psi_j$, $k=k_{j}$ (resp. $v=\psi_0$, $k=k_{0}$). We recall from \eqref{identity1} that
\begin{equation} \label{fedi2}
4\pi\,\Im m\, (T(k_j)\psi_j, \psi_j)_{\Om} =  k_j \int_{\mathbb{S}^{2}} |\herg(k_j)^{\ast}T(k_j)\psi_j|^2 ds.
\end{equation}
Since $\herg(k_j)^{\ast}T(k_j):\mL^2(\Om)\to\mL^2(\mathbb{S}^2)$ is compact (because $T(k_j):\mL^2(\Om)\to\mL^2(\Om)$ is continuous and because $\herg(k_j)^{\ast}:\mL^2(\Om)\to\mL^2(\mathbb{S}^2)$ is compact) and since $k\mapsto \herg(k)^{\ast}T(k)$ is continuous in the operator norm, we deduce that
\[
\begin{array}{l}
4\pi\,\Im m\, (T(k_j)\psi_j, \psi_j)_{\Om} =  \dsp k_j \int_{\mathbb{S}^{2}} |\herg(k_j)^{\ast}T(k_j)\psi_j|^2 ds\\[10pt]
\hspace{3cm}\dsp\underset{j\to+\infty}{\to} k_0 \int_{\mathbb{S}^{2}} |(\herg(k_0))^{\ast}T(k_0)\psi_0|^2 ds=4\pi\,\Im m\, (T(k_0)\psi_0, \psi_0)_{\Om}.
\end{array}
\]
From \eqref{fedi1} then we get $\Im m\, (T(k_0)\psi_0, \psi_0)_{\Om} =0$ and therefore
$(\psi_0 , w_0)$ is a solution of the interior transmission problem \eqref{NewITEP} for $k=k_0$.  The hypothesis on $k_0$ implies $\psi_0 =0$. Using the definition of $T(k)$ we have 
$$
 \frac{k^2_j }{4\pi} ( (n-n_b(k_j)) \psi_j,{ \psi_j})_{\Om} =  ({ T(k_j)
  \psi_j},{ \psi_j})_{\Om}  - \frac{k^2_j }{4\pi} ( (n-n_b(k_j)) \psi_j,{ w_j})_{\Om}
$$
where $( (n-n_b(k_j)) \psi_j,{ w_j})_{\Om} \to ( (n-n_b(k_0)) \psi_0,{
  w_0})_{\Om}$ when $j\to+\infty$. The latter property is a consequence of Proposition \ref{PropoEstimRegu} and the fact that $\mH^2(\Om)$ is compactly embedded in $\mL^2(\Om)$. Consequently 
$$
 0 \le \frac{k^2_j }{4\pi} ( (n-n_b(k_j)) \psi_j,{ \psi_j})_{\Om} \underset{j\to+\infty}{\to} -1
$$
which is a contradiction.\\
\newline
We now proceed with the proof of the second part of the theorem related to the convergence of the sequence $(v_j)$. Let $v$ be the weak limit in $\mL^2(\Om)$ of a subsequence of $(v_j)$. Note that this limit exists because $\|v_j\|_{\mL^2(\Om)}=1$ for all $j\in\N$. To simplify, the subsequence is also denoted $(v_j)$. We have 
\begin{equation} \label{fedi3}
({ T(k_j) v_j},{ v_j})_{\Om} = \theta_j  \,\cfrac{\lambda_{\star}(k_j)}{|\lambda_{\star}(k_j)|} 
\end{equation}
with $\theta_j:=|\lambda_{\star}(k_j)|/{\|\herg(k_j) g_{j} \|^2_{\mL^2(\Om)}}$. Since there holds
\[
\theta_j=\cfrac{|\lambda_{\star}(k_j)|}{\|\herg(k_j) g_{j} \|^2_{\mL^2(\Om)}}=\cfrac{(T(k_{j})\herg(k_j) g_{j},\herg(k_j) g_{j})_{\Om}}{\|\herg(k_j) g_{j} \|^2_{\mL^2(\Om)}},
\]
using that $k\mapsto T(k)$ is continuous in the operator norm, we infer that $(\theta_j)$ is bounded. Therefore, up to changing the subsequence, one can assume that $\lim_{j\to+\infty}\theta_j =\theta_0\ge0$. Observing from (\ref{fedi3}) that 
$$
\lim_{j\to+\infty}\Im m\,({ T(k_j) v_j},{ v_j})_{\Om} =0
$$
and using the same arguments as above we conclude that the pair $(v, w)$ solves the problem \eqref{NewITEP} for $k=k_0$, $w$ being the solution of \eqref{eqwart} with $k=k_0$. Now we prove that $v \not\equiv 0$ and that $v$ is the strong limit of the sequence $(v_j)$ in $\mL^2(\Om)$. We start again from the identity
$$
  ( (n-n_b(k_j)) v_j,{ v_j})_{\Om} =  \frac{4\pi}{k^2_j } ({ T(k_j)
  v_j},{ v_j})_{\Om}  - ( (n-n_b(k_j)) v_j,{ w_j})_{\Om}.
$$
Formula (\ref{fedi3}) ensures that the sequence $(4\pi k^{-2}_j ({ T(k_j) v_j},{ v_j})_{\Om})$ (up to extraction of a subsequence) converges to a non positive real number. We deduce that 
\begin{equation}\label{identityITEP}
 \limsup_{j\to+\infty}\ ( (n-n_b(k_j)) v_j,{ v_j})_{\Om} \le   - ( (n-n_b(k_0)) v, w)_{\Om}
\end{equation}
where we used the strong convergence of $w_j$ to $w$ in $\mL^2(\Om)$ (again consequence of Proposition \ref{PropoEstimRegu}) and the continuity of $n_b(k)$ with respect to $k$. One easily observes that since the pair $(v, w)$ solves \eqref{NewITEP}, we have
\[
( (n-n_b(k_0))(v+w), v)_{\Om} =0.
\]
Using this identity in (\ref{identityITEP}), we get 
$$
 \limsup_{j\to+\infty}\ ( (n-n_b(k_j)) v_j,{ v_j})_{\Om} \le   ( (n-n_b(k_0)) v, v)_{\Om}.
$$
We then obtain, since $( (n_b(k_j)-n_b(k_0)) v, v)_{\Om} \to 0$ as $j \to \infty$,
$$
  \limsup_{j\to+\infty}\ ( (n-n_b(k_j))(v_j-v),{ v_j}-v)_{\Om} \le  0,
$$
which is enough to conclude that $(v_j)$ converges to $v$ in $\mL^2(\Om)$. Since $\|v_{j}\|_{\mL^2(\Om)}=1$ for all $j\in\N$, we have $v\not\equiv0$.
\end{proof}

\begin{remark}
In Theorem \ref{TheoPhase1}, one can replace $k$ with $\eps$ and fix $k=k_0$ where $\eps$ is a parameter such that $T_\eps$ converges to $T_0$ in operator norm. This can be interesting for example in a situation where $T_\eps$ is associated with a coefficient $n_\eps$ such that $\| n_\eps - n\|_{\mL^\infty(\R^3)} \to 0$. Unfortunately, this does not cover situations where we have only $\| n_\eps - n\|_{\mL^2(\R^3)} \to 0$ which would allows us for example to deal with the interesting case with two separated inclusions touching at the limit.
\end{remark}
\noindent Theorem \ref{TheoPhase1} only gives a sufficient condition to ensure that a value $k_0$ is a TE. In general, it is hard to establish that for all TEs, we have a sequence of frequencies $(k_{j})$ converging to $k_0$ and such that $\lim_{j\to+\infty}\delta_{\star}(k_j)=2\pi$ (resp. $\lim_{j\to+\infty}\delta_{\star}(k_j)=0$) when $n-n_b \ge \alpha >0$ (resp. $n_b-n \ge \alpha >0$)  in $\Om$. For example, in \cite{KiLe13}, where the case $n_b=1$ is treated, this is proved only for the first eigenvalue under quite restrictive assumptions on the index material $n$ which has to be constant and large or small enough. 
We shall prove in the sequel that for artificial backgrounds that satisfies \eqref{choixdenb}  we have the above characterisation for all TEs leading to an ``if and only if statement''.

\subsection{Artificial backgrounds leading to a necessary condition}
In order to prove the converse statement of Theorem \ref{TheoPhase1}, we first introduce some material. If $k$ is not a TE, $\mathscr{F}$ is injective and has dense range in $\mL^2(\mathbb{S}^2)$. In that case, $1$ is not an eigenvalue of $\mathscr{S}$. Following \cite{KiLe13}, we denote $\mathfrak{S}$ the Cayley transform of $\mathscr{S}$ defined by 
\[
\mathfrak{S}=i(\mrm{Id}+\mathscr{S})(\mrm{Id}-\mathscr{S})^{-1}:R(\mathscr{F})\subset \mL^2(\mathbb{S}^2)\to\mL^2(\mathbb{S}^2).
\]
Here $R(\mathscr{F})$ stands for the range of $\mathscr{F}$. The operator $\mathfrak{S}$ is selfadjoint and its spectrum is discrete. Moreover $e^{i\delta_j}$, with $\delta_j\in(0;2\pi)$, is an eigenvalue of $\mathscr{S}$ if and only if $-\cot(\delta_j/2)\in\R$ is an eigenvalue of $\mathfrak{S}$ (see the connection between the quantities in Figure \ref{Illustration2}).\\

\begin{figure}[!ht]
\centering
\begin{tikzpicture}[scale=1]
\begin{scope}
    \clip(-2.1,-2) rectangle (7,2);
\draw[->] (-0.2,0) -- (8,0);
\draw[->] (0,-3) -- (0,4.2);
\draw[dotted,gray,line width=0.3mm] ({6.5*exp(-5/4)},0) -- ({6.5*exp(-5/4)},-0.7415375);
\draw[dotted,gray,line width=0.3mm] (0,-0.7415375) -- ({6.5*exp(-5/4)},-0.7415375);
\draw[dotted,gray,line width=0.3mm] (0,{-cos((6.5*exp(-1/4))*57.3/2)/sin((6.5*exp(-1/4))*57.3/2)}) -- ({6.5*exp(-1/4)},{-cos((6.5*exp(-1/4))*57.3/2)/sin((6.5*exp(-1/4))*57.3/2)});
\draw[dotted,gray,line width=0.3mm] ({6.5*exp(-1/4)},0) -- ({6.5*exp(-1/4)},1.43414615);
\draw[gray,line width=0.3mm] plot[domain=0.5:5.6,samples=100] (\x,{-cos(\x*57.3/2)/sin(\x*57.3/2)});
\draw (6.28,-0.1)--(6.28,0.1);
\node at({6.5*exp(-1/4)},-0.4){$\delta_{\star}$};
\node at(-1.2,1.4){$-\cot(\delta_{\star}/2)$};
\node at({6.5*exp(-5/4)},0.4){$\delta_{j}$};
\node at(-1.2,-0.76){$-\cot(\delta_{j}/2)$};
\node at(-0.2,-0.2){$0$};
\node at(6.28,-0.4){$2\pi$};
\foreach \x in {1,...,20}
\draw[draw=blue,thick] ({6.5*exp(-\x/4)},{-0.1})--({6.5*exp(-\x/4)},{0.1});
\foreach \x in {1,...,20}
\pgfmathtruncatemacro{\tmp}{110/\x}
\draw[draw=red,thick] ({-0.1},{-cos((6.5*exp(-\x/4))*57.3/2)/sin((6.5*exp(-\x/4))*57.3/2)})--({0.1},{-cos((6.5*exp(-\x/4))*57.3/2)/sin((6.5*exp(-\x/4))*57.3/2)});
\end{scope}
\end{tikzpicture}
\caption{Connection between the phases $\delta_j$ of the eigenvalues of $\mathscr{S}$ and the eigenvalues of $\mathfrak{S}$ (the $-\cot(\delta_j/2)$). Here the representation corresponds to a situation where $n-n_b\ge\alpha>0$ in $\Om$. \label{Illustration2}}
\end{figure}
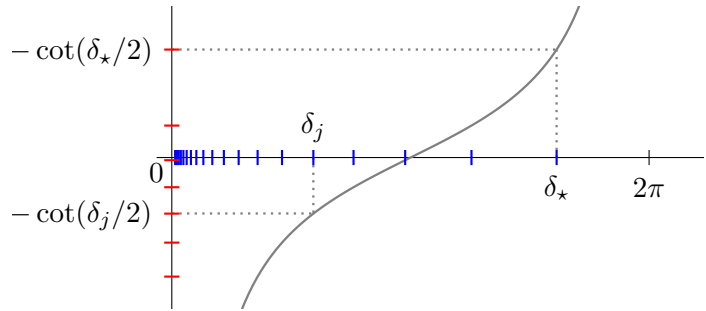
\noindent We now recall the following result from \cite[Lemma 4.3]{KiLe13}. We here reproduce its short and elegant proof for the reader's convenience. Since this result holds for fixed $k$, we do not indicate the dependence on $k$ in the notation.
\begin{proposition}\label{PropoInfSup}
Assume that $k>0$ is not a TE of (\ref{NewITEP}).\\[4pt]
$i)$ Assume that we have $n-n_b \ge \alpha >0$ in $\Om$. Then 
\begin{equation}\label{quotient1}
\cot\cfrac{\delta_{\star}}{2} = \inf_{\varphi\in\mrm{H}_{\mrm{inc}}}\cfrac{\Re e\,(T\varphi,\varphi)_{\Om}}{\Im m\,(T\varphi,\varphi)_{\Om}}.
\end{equation}
$ii)$ Assume that we have $n_b-n \ge \alpha >0$ in $\Om$. Then 
\begin{equation}\label{quotient2}
\cot\cfrac{\delta_{\star}}{2} = \sup_{\varphi\in\mrm{H}_{\mrm{inc}}}\cfrac{\Re e\,(T\varphi,\varphi)_{\Om}}{\Im m\,(T\varphi,\varphi)_{\Om}}.
\end{equation}
We emphasize that the denominators in (\ref{quotient1}), (\ref{quotient2}) do not vanish when $k>0$ is not a TE of (\ref{NewITEP}) (see Section \ref{sectionFirstCharac}).
\end{proposition}
\begin{proof}
Let us consider the situation $n-n_b \ge \alpha >0$ in $\Om$. Since $\mathfrak{S}$ is selfadjoint, we can apply the Courant-Fischer $\inf$-$\sup$ principle to get
\[
\begin{array}{lcl}
-\cot\cfrac{\delta_{\star}}{2} &=& \dsp\sup_{f\in R(\mathscr{F})}\cfrac{(\mathfrak{S}f,f)_{\mathbb{S}^2}}{\|f\|_{\mL^2(\mathbb{S}^2)}}\ =\ \sup_{f\in R(\mathscr{F})}\cfrac{(i(\mrm{Id}+\mathscr{S})(\mrm{Id}-\mathscr{S})^{-1}f,f)_{\mathbb{S}^2}}{\|f\|_{\mL^2(\mathbb{S}^2)}}\\[14pt]
& = & \dsp\sup_{g\in \mL^2(\mathbb{S}^2)}\cfrac{(i(\mrm{Id}+\mathscr{S})g,(\mrm{Id}-\mathscr{S})g)_{\mathbb{S}^2}}{\|(\mrm{Id}-\mathscr{S})g\|_{\mL^2(\mathbb{S}^2)}}\\[14pt]
& = & \dsp\sup_{g\in \mL^2(\mathbb{S}^2)}\cfrac{i(\|g\|^2_{\mL^2(\mathbb{S}^2)}+2i\Im m\,(\mathscr{S}g,g)_{\mathbb{S}^2}-\|\mathscr{S}g\|^2_{\mL^2(\mathbb{S}^2)})}{\|g\|^2_{\mL^2(\mathbb{S}^2)}-2\Re e\,(\mathscr{S}g,g)_{\mathbb{S}^2}+\|\mathscr{S}g\|^2_{\mL^2(\mathbb{S}^2)}}\\[14pt]
& = & \dsp\sup_{g\in \mL^2(\mathbb{S}^2)}\cfrac{\Im m\,(\mathscr{S}g,g)_{\mathbb{S}^2}}{\Re e\,(\mathscr{S}g,g)_{\mathbb{S}^2}-\|g\|^2_{\mL^2(\mathbb{S}^2)}}.
\end{array}
\]
Using the fact that $\mathscr{S}=\mrm{Id}+\frac{ik}{2\pi}\mathscr{F}$ and that $\mathscr{F}=\herg^{\ast}T\herg$ (Proposition \ref{PropFunitary}), we can write
\[
-\cot\cfrac{\delta_{\star}}{2} = \dsp\sup_{g\in \mL^2(\mathbb{S}^2)}\cfrac{\Re e\,(\mathscr{F}g,g)_{\mathbb{S}^2}}{-\Im m\,(\mathscr{F}g,g)_{\mathbb{S}^2}}= \dsp\sup_{g\in \mL^2(\mathbb{S}^2)}\cfrac{\Re e\,(T\herg g,\herg g)_{\mathbb{S}^2}}{-\Im m\,(T\herg g,\herg g)_{\mathbb{S}^2}}=\dsp\sup_{\varphi\in\mrm{H}_{\mrm{inc}}}\cfrac{\Re e\,(T\varphi,\varphi)_{\Om}}{-\Im m\,(T\varphi,\varphi)_{\Om}}.
\]
This proves (\ref{quotient1}). When $n_b-n\ge\alpha>0$ in $\Om$, the identity (\ref{quotient2}) can be shown in a similar way.
\end{proof}
\noindent To continue the analysis, we use again the $k$-dependent notation. If $k_0$ is a TE of (\ref{NewITEP}), denote $(v,w)\in\mL^2(\Om)\times\mH^2_0(\Om)$ a corresponding eigenpair. Then $v$ belongs to $\mrm{H}_{\mrm{inc}}(k_0)$ and from Proposition \ref{propoTET}, we know that we have $(T(k_0)v,v)_{\Om}=0$. Now we compute a Taylor expansion of $k\mapsto(T(k)v,v)_{\Om}$ as $k\to k_0$. This will be useful to assess the right hand sides of (\ref{quotient1}), (\ref{quotient2}). 

\begin{proposition} \label{TheoPhase3} 
Assume that $n_b$ satisfies \eqref{choixdenb}, i.e. $n_b(k) = \rho/k^2$ in $\Om$ with $\rho \in \R$ independent from $k$.
Let $k_0>0$ be a TE of (\ref{NewITEP}) and $(v_0,w_0)\in\mL^2(\Om)\times\mH^2_0(\Om)$ an associated eigenpair. Then there is $\eps>0$ such that we have the expansion
\begin{equation}\label{DesiredExpansion}
4\pi (T(k) v_0, v_0)_{\Om} = 0+2k_0(k-k_0) (n  (v_0 + w_0), (v_0 + w_0))_{\Om}  + (k-k_0)^2 \eta(k),
\end{equation}
where the remainder $\eta(k)$ satisfies $|\eta(k)| \le C \|v_0\|^2_{\mL^2(\Om)}$ with $C>0$ independent from $k\in[k_0-\eps;k_0+\eps]$.
\end{proposition}
\begin{proof}
According to the definition of $T(k)$ in (\ref{defTtC}), we have 
\begin{equation}\label{FirstLine}
{4\pi} (T(k) v_0,v_0)_{\Om} =k^2((n-n_b(k))(v_0+w(k)),v_0)_{\Om}
\end{equation}
where $w(k)$ is the solution of \eqref{eqwart} with $v=v_0$. We remark that, according to the definition of TEs,  the solution $w(k_0)$ of \eqref{eqwart} with $v=v_0$ and $k=k_0$ is such that $w(k_0) =w_0$ in $\Om$ and $w(k_0) =0$ outside $\Om$. We first need to compute the derivative $w'$ of $w$ at $k=k_0$. To proceed, we prove an expansion as $k\to k_0$ of the form
\begin{equation}\label{Expansion}
w(k)-w(k_0)=(k-k_0)w'+(k-k_0)^2\tilde{w}(k),
\end{equation}
where $w'$ is independent from $k$ and where $\tilde{w}(k)$ have bounded norm as $k \to k_0$. Introduce some $R$ large enough so that $\overline{\Om}\subset B_R$ and consider  the Dirichlet-to-Neumann operator $\Lambda(k):\mH^{1/2}(\partial B_R)\to\mH^{-1/2}(\partial B_R)$ such that $\Lambda(k)\varphi=\partial_\nu\psi$ ($\nu$ is oriented to the exterior of $B_R$) where $\psi\in\mH^1_{\loc}(\R^3\setminus B_R)$ is the outgoing function solving 
\begin{equation}\label{DefDTN}
\begin{array}{|rcll}
\Delta\psi+k^2\psi&=&0 & \mbox{ in }\R^3\setminus B_R\\[4pt]
\psi&=&\varphi& \mbox{ on  }\partial B_R.
\end{array}
\end{equation}
Since $k^2n_b(k)=\rho$ in $\Om$, the functions $w(k)$, $w(k_0)$ satisfy, for all $\varphi\in\mH^1(B_R)$,
\begin{equation}\label{TwoLines}
\begin{array}{l}
(\nabla w(k),\nabla \varphi)_{\Om}-k^2(nw(k),\varphi)_{\Om}-\langle \Lambda(k)w(k),\varphi\rangle_{\partial B_R}=((k^2n-\rho)v_0,\varphi)_{\Om},\\[4pt]
(\nabla w(k_0),\nabla \varphi)_{\Om}-k_0^2(nw(k_0),\varphi)_{\Om}-\langle \Lambda(k)w(k_0),\varphi\rangle_{\partial B_R}=((k_0^2 n-\rho)v_0,\varphi)_{\Om}
\end{array}
\end{equation}
where $\langle \cdot,\cdot\rangle_{\partial B_R}$ denotes the $\mH^{-1/2}(\partial B_R)-\mH^{1/2}(\partial B_R)$ duality product. Remark that we used that $w(k_0) =0$ outside $\Om$ to replace $\Lambda(k_0)$ with $\Lambda(k)$ in the second equation.
Computing the difference of the two lines of (\ref{TwoLines}) and taking the limit as $k\to k_0$, one finds that $w'\in\mH^2_{\loc}(\R^3)$ satisfies, for all $\varphi\in\mH^1(B_R)$,
\[
(\nabla w',\nabla \varphi)_{\Om}-k^2_0(nw',\varphi)_{\Om}-\langle \Lambda(k_0)w',\varphi\rangle_{\partial B_R}=2k_0(n (v_0+w_0),\varphi)_{\Om}.
\]
Then one obtains that $\tilde{w}(k)\in\mH^2_{\loc}(\R^3)$ solves,  for all $\varphi\in\mH^1(B_R)$,
\[
\begin{array}{lcl}
(\nabla \tilde{w}(k),\nabla \varphi)_{\Om}-k^2(n\tilde{w}(k),\varphi)_{\Om}-\langle \Lambda(k)\tilde{w}(k),\varphi\rangle_{\partial B_R}&=&(n(v_0+w_0)+(k+k_0)w',\varphi)_{\Om}\\[10pt]
 & & -\langle\ \cfrac{\Lambda(k)-\Lambda(k_0)}{k-k_0}\,w',\varphi\rangle_{\partial B_R}.
\end{array}
\]
Using that the mapping $k \mapsto \Lambda(k)$ is real analytic from $\R^{\ast}_+$ into $\mathcal{L}(\mH^{1/2}(\partial B_R),\mH^{-1/2}(\partial B_R))$, we obtain (using the uniform bounds for scattering problems as in Proposition \ref{PropoEstimRegu})  that there holds $\|\tilde{w}(k) \|_{\mH^2(\Om)} \le C \|v\|_{\mL^2(\Om)}$ for some constant $C$ independent from $k\in I$, $I$ being a given compact set of $\R^{\ast}_+$. Now inserting the expansion  (\ref{Expansion}) in (\ref{FirstLine}) and using that $(T(k_0) v_0 ,v _0)_{\Om} =0$, we get 
\begin{equation}\label{Expansion1}
\frac{4\pi}{ (k-k_0)}(T(k) v_0, v_0)_{\Om}  =2k_0 (n(w_0 + v_0), v_0)_{\Om}  + ((k_0^2n-\rho) w', v_0)_{\Om} + (k-k_0) \eta(k)
\end{equation}
with 
$$
\eta(k) = \left((k^2n-\rho) \tilde w(k) + (k+k_0)n w'+ n  (w_0 + v_0), v_0\right)_{\Om}.
$$
Obviously, the reminder $\eta(k) $ satisfies the uniform estimate indicated in the Proposition. We recall that $w_0 \in \mH^2_0(\Om)$ and satisfies
\begin{equation}\label{hint1}
\Delta w_0 + k_0^2 n w_0 = - (k^2_0 n- \rho) v_0.
\end{equation}
This allows us to write, since $n$ and $\rho$ are real,
\[ \begin{array}{ll}
((k^2_0 n- \rho) w', v_0)_{\Om} = (w', (k^2_0 n- \rho) v_0)_{\Om} &= - ( w', \Delta w_0 + k_0^2 n w_0)_{\Om} \\ &
 = - ( \Delta w' + k_0^2 n w', w_0)_{\Om} = 2k_0 ( n( v_0 + w_0), w_0 )_{\Om}.
\end{array}\]
Using the latter identity in (\ref{Expansion1}), finally we obtain the desired expansion (\ref{DesiredExpansion}).
\end{proof}
\noindent Denote by $\lim_{k\nearrow k_0}$ (resp. $\lim_{k\searrow k_0}$) the limit as $k\to k_0$ with $k<k_0$ (resp. $k>k_0$). Assume that $k_0$ is a TE of (\ref{NewITEP}) and that $(v_0,w_0)\in\mL^2(\Om)\times\mH^2_0(\Om)$ is a corresponding eigenpair. According to identity (\ref{identity1}), there holds $\Im m\,(T(k)v_0,v_0)_{\Om}>0$ when $k$ is not a TE. Using the expansion (\ref{DesiredExpansion}), we deduce that we have
\begin{equation}\label{LineLim}
\cfrac{\Re e\,(T(k)v_0,v_0)_{\Om}}{\Im m\,(T(k)v_0,v_0)_{\Om}}=\cfrac{2k_0(n  (v_0 + w_0), (v_0 + w_0))_{\Om}+(k-k_0)\,\Re e\,\eta(k)}{(k-k_0)\,\Im m\,\eta(k)}\underset{k\nearrow k_0}{\to}-\infty.
\end{equation}
Note that we have $w_0+v_0 \not\equiv0$ in $\Om$ otherwise $w_0\in\mH^2_0(\Om)$ would be null and $v_0$ too. In order to obtain the converse statement of Theorem \ref{TheoPhase1}, we now can apply the result of Proposition \ref{PropoInfSup}. The important point in here, is that for  $n_b$ satisfying \eqref{choixdenb}, i.e. $n_b(k) = \rho/k^2$ in $\Om$ with $\rho$ independent from $k$, the space
$\mrm{H}_{\mrm{inc}}(k)$ is in fact independent from $k$ since \[
\mrm{H}_{\mrm{inc}}(k)=\mrm{H}_{\mrm{inc}} =\{\varphi \in \mL^2(\Om); \; \Delta\varphi+\rho \varphi = 0 \mbox{ in } \Om\}.
\]
We refer the reader to the next paragraph for discussing the case where $n_b$ is constant independent from $k$. 

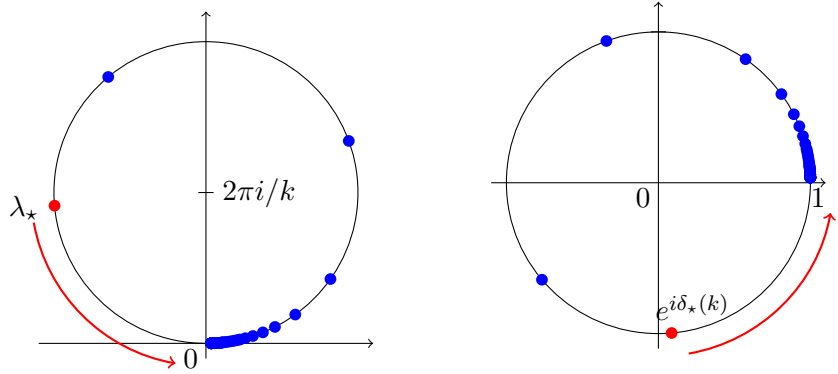
\begin{figure}[!ht]
\centering
\begin{tikzpicture}[scale=1]
\draw (0,2) circle (2);
\draw[->] (0,-0.2)--(0,4.4);
\draw[->] (-2.2,0)--(2.2,0);
\draw (-0.1,2)--(0.1,2);
\node at(0.7,2){$2\pi i/k$};
\node at(-2.4,1.8){$\lambda_\star$};
\node at(-0.2,-0.2){$0$};
\foreach \x in {0.3,0.5,1,2,...,50}
\pgfmathtruncatemacro{\tmp}{110/\x}
\draw[fill=blue,draw=none] ({2*cos(\tmp-90)},{2+2*sin(\tmp-90)}) circle (2.2pt);
\pgfmathtruncatemacro{\tmp}{110/0.4}
\draw[fill=red,draw=none] ({2*cos(\tmp-90)},{2+2*sin(\tmp-90)}) circle (2.2pt);
\begin{scope}[yshift=2cm]
\draw[->,red,thick] (190:2.3) arc (190:260:2.3) ;
\end{scope}
\end{tikzpicture}\qquad\qquad 
\raisebox{0.3cm}{\begin{tikzpicture}[scale=1]
\draw (0,0) circle (2);
\draw[->] (0,-2.2)--(0,2.4);
\draw[->] (-2.2,0)--(2.2,0);
\draw (-0.1,2)--(0.1,2);
\node at(0.45,-1.65){$e^{i\delta_{\star}(k)}$};
\node at(-0.2,-0.2){$0$};
\node at(2.1,-0.2){$1$};
\draw[->,red,thick] (280:2.3) arc (280:350:2.3) ;
\foreach \x in {0.3,0.5,1,2,...,50}
\pgfmathtruncatemacro{\tmp}{110/\x}
\draw[fill=blue,draw=none] ({2*cos(\tmp)},{2*sin(\tmp)}) circle (2.2pt);
\pgfmathtruncatemacro{\tmp}{110/0.4}
\draw[fill=red,draw=none] ({2*cos(\tmp)},{2*sin(\tmp)}) circle (2.2pt);
\end{tikzpicture}}
\caption{Illustration of the behaviours of $k\mapsto \lambda_{\star}(k)$ (left) and $k\mapsto e^{i\delta_{\star}(k)}$ (right) when $k\nearrow k_0$ where $k_0$ is a TE of (\ref{NewITEP}). Here we consider a situation where $n-n_b(k_0)\ge\alpha>0$ in $\Om$ with $n_b(k)=\rho/k^2$.\label{Illustration3}}
\end{figure}

\noindent We now state and prove the second main result of the article.
\begin{theorem} \label{TheoPhase2}
Assume that $n_b(k)=\rho/k^2$ in $\Om$ for some $\rho\in\R$. Let $k_0>0$ be a TE of (\ref{NewITEP}). \\[3pt]
$i)$ Assume that we have $n-n_b(k_0) \ge \alpha >0$ in $\Om$. Then 
\[
\hspace{1.5cm}\lim_{k\nearrow k_0}\delta_{\star}(k)=2\pi \qquad\mbox{(see Figure \ref{Illustration3})}.
\]
$ii)$ Assume that we have $n_b(k_0)-n \ge \alpha >0$ in $\Om$. Then 
\[
\hspace{1.5cm}\lim_{k\searrow k_0}\delta_{\star}(k)=0.\phantom{\qquad\mbox{(see Figure \ref{Illustration3})}}
\]
\end{theorem}
\begin{proof}
Let us consider the case $n-n_b(k_0) \ge \alpha >0$ in $\Om$. Let $k_0$ be a TE of (\ref{NewITEP}) and $(v_0,w_0)\in\mL^2(\Om)\times\mH^2_0(\Om)$ a corresponding eigenpair. According to Proposition \ref{PropoInfSup}, we have 
\[
\cot\cfrac{\delta_{\star}(k)}{2} = \inf_{\varphi\in\mrm{H}_{\mrm{inc}}}\cfrac{\Re e\,(T(k)\varphi,\varphi)_{\Om}}{\Im m\,(T(k)\varphi,\varphi)_{\Om}} \le \cfrac{\Re e\,(T(k)v_0,v_0)_{\Om}}{\Im m\,(T(k)v_0,v_0)_{\Om}}.
\]
Then from (\ref{LineLim}), which is a consequence of Proposition \ref{TheoPhase3}, we infer that 
\[
\lim_{k\nearrow k_0}\cot\cfrac{\delta_{\star}(k)}{2}=-\infty\qquad\mbox{ and so }\qquad \lim_{k\nearrow k_0}\delta_{\star}(k)=2\pi.
\]
The case $n_b(k_0)-n \ge \alpha >0$ in $\Om$ can be proved similarly.
\end{proof}

\begin{remark}
We emphasize that the case of the ZIM background, that is when $n_b=0$ in $\Om$ and $n_b=1$ in $\R^3\setminus\overline{\Om}$ (see the discussion in (\ref{NewITEPZIM})) is covered by the first statement of Theorem \ref{TheoPhase2}. We also remark that choosing $\rho \le 0$ ensures that we are always in the configuration illustrated by Figure \ref{Illustration3}.
\end{remark}

\subsection{Some remarks on the case $n_b$ independent from $k$}
We here give some indications on the difficulties encountered when making ``the natural choice'' of $n_b$ independent from $k$. In this case one can obtain a similar expansion as in Proposition \ref{TheoPhase3}. 
\begin{proposition} \label{TheoPhase3b} 
Assume that $n_b$ is independent from $k$. Let $k_0>0$ be a TE of (\ref{NewITEP}) and $(v_0,w_0)\in\mL^2(\Om)\times\mH^2_0(\Om)$ an associated eigenpair. Then there is $\eps>0$ such that we have the expansion
\begin{equation}\label{DesiredExpansionb}
4\pi (T(k) v_0, v_0)_{\Om} = 0+\frac{2}{k_0} (k-k_0)  \|\nabla w_0\|_{\mL^2(\Om)}^2 + (k-k_0)^2 \eta(k),
\end{equation}
where the remainder $\eta(k)$ satisfies $|\eta(k)| \le C \|v\|^2_{\mL^2(\Om)}$ with $C>0$ independent from $k\in[k_0-\eps;k_0+\eps]$.
\end{proposition}
\begin{proof}
The proof follows the same lines as the proof of Theorem \ref{TheoPhase3}. We here just sketch the main steps. From the definition of $T(k)$ we have 
\begin{equation}\label{FirstLineb}
{4\pi} (T(k) v_0,v_0)_{\Om} =k^2((n-n_b)(v_0+w(k)),v_0)_{\Om}.
\end{equation}
We also have an expansion similar to \eqref{Expansion}
\begin{equation}\label{Expansionb}
w(k)-w(k_0)=(k-k_0)w'+(k-k_0)^2\tilde{w}(k),
\end{equation}
where the derivative $w'$ now satisfies  $w'\in\mH^2_{\loc}(\R^3)$ 
\[
(\nabla w',\nabla \varphi)_{\Om}-k^2_0(nw',\varphi)_{\Om}-\langle \Lambda(k_0)w',\varphi\rangle_{\partial B_R}=2k_0((n-n_b)v_0+n w(k_0),\varphi)_{\Om}
\]
 for all $\varphi\in\mH^1(B_R)$. The reminder  $\tilde{w}(k)\in\mH^2_{\loc}(\R^3)$ solves,  for all $\varphi\in\mH^1(B_R)$,
\[
\begin{array}{lcl}
(\nabla \tilde{w}(k),\nabla \varphi)_{\Om}-k^2(n\tilde{w}(k),\varphi)_{\Om}-\langle \Lambda(k)\tilde{w}(k),\varphi\rangle_{\partial B_R}&=&((n-n_b)v_0+n w(k_0)+(k+k_0)w',\varphi)_{\Om}\\[10pt]
 & & -\langle\ \cfrac{\Lambda(k)-\Lambda(k_0)}{k-k_0}\,w',\varphi\rangle_{\partial B_R}
\end{array}
\]
 and therefore verifies the same uniform bound: $\|\tilde{w}(k) \|_{\mH^2(\Om)} \le C \|v\|_{\mL^2(\Om)}$ for some constant $C$ independent from $k\in I$, $I$ being a given compact set of $\R^{\ast}_+$. Now inserting the expansion  (\ref{Expansionb}) in (\ref{FirstLineb}) and using that $(T(k_0) v_0 ,v_0 )_{\Om} =0$, we get 
\begin{equation}\label{Expansion1b}
\frac{4\pi}{k^2} (T(k) v_0, v_0)_{\Om}  = (k-k_0)((n-n_b) w', v_0)_{\Om} + (k-k_0)^2( \tilde{w}(k),v_0)_{\Om}.
\end{equation}
Since $n$ and $n_b$ are real, and $w_0 \in \mH^2_0(\Om)$ with $\Delta w_0 + k_0^2 n_b w_0 = k^2_0 (n_b-n)v_0$, then
\[
k_0^2((n-n_b) w', v_0)_{\Om} = - ( w', \Delta w_0 + k_0^2 n w_0)_{\Om}  = - ( \Delta w' + k_0^2 n w', w_0)_{\Om} = 2k_0 ( (n-n_b) v_0 + n w, w )_{\Om}.
\]
On the other hand, we also have $\Delta w_0 + k_0^2 n_b w_0 = k^2_0 (n_b-n)(v_0+w_0)$ which implies
$$
k_0^2( (n-n_b) v_0 + n w_0, w_0 )_{\Om} = \|\nabla w_0\|^2_{\mL^2(\Om)}.
$$
We deduce
\[
k_0^2((n-n_b) w', v_0)_{\Om} = \cfrac{2}{k_0}\, \|\nabla w_0\|^2_{\mL^2(\Om)^3}.
\]
Using the latter identity in (\ref{Expansion1b}), one obtains the desired expansion (\ref{DesiredExpansionb}).
\end{proof}

\noindent Notice that we obtain in Proposition \ref{TheoPhase3b} the same sign for the leading term of the expansion of $(T(k) v_0, v_0)_{\Om}$ as in Proposition \ref{TheoPhase3}. Notice also that the results coincide for the case $n_b =0$ since in that case $\Delta w_0 = -n k_0^2 (v_0 + w_0)$ and $\Delta v_0 = 0$ in $\Om$. Therefore $ \|\nabla w_0\|^2_{\mL^2(\Om)} = k_0^2 
(n (v_0 + w_0), w_0)_{\Om}$ and  $k_0^2 (n (v_0 + w_0), v_0)_{\Om} = 0$.\\
\newline
The difficulty in exploiting the result of Proposition \ref{TheoPhase3b} lies in the fact that $v_0$ which is in $\mrm{H}_{\mrm{inc}}(k_0)$ does not belong to $\mrm{H}_{\mrm{inc}}(k)$ when $k\neq k_0$ and $n_b\ne0$ in $\Om$. This is a problem because with the $k$-dependent notation, formula (\ref{quotient1}) (the same is true for (\ref{quotient2})) writes
\begin{equation}\label{Quotientkdep}
\cot\cfrac{\delta_{\star}(k)}{2} = \inf_{\varphi\in\mrm{H}_{\mrm{inc}}(k)}\cfrac{\Re e\,(T(k)\varphi,\varphi)_{\Om}}{\Im m\,(T(k)\varphi,\varphi)_{\Om}}
\end{equation}
when $k$ is not a TE of (\ref{NewITEP}). Therefore in general we cannot  simply take $\varphi=v_0$ in (\ref{Quotientkdep}). In \cite{KiLe13}, when dealing with the case $n_b=1$, the authors took $\varphi=P(k)v_0$ with $P(k)v_0$ equal to the $\mL^2(\Om)$ projection of $v_0$ in $\mrm{H}_{\mrm{inc}}(k)$. But then it is necessary to compute the Taylor expansion of $k\mapsto(T(k)P(k)v_0,P(k)v_0)_{\Om}$ as $k\to k_0$ whose expression does not allow one to conclude simply that
\begin{equation}\label{ResultKiLe}
\lim_{k\nearrow k_0}\cfrac{\Re e\,(T(k)P(k)v_0,P(k)v_0)_{\Om}}{\Im m\,(T(k)P(k)v_0,P(k)v_0)_{\Om}}=-\infty.
\end{equation}
Actually in \cite{KiLe13}, the authors have been able to obtain (\ref{ResultKiLe}) only for the first (classical) TE and for indices $n$ which are constant and large or small enough. 

\section{Numerical experiments: the case of the disk}\label{SectionNum}

In this section, we illustrate the results of Theorems \ref{TheoPhase1} and \ref{TheoPhase2} with simple 2D numerical examples. In the initial problem (\ref{PbChampTotalFreeSpaceIntrod}), we take $n$ equal to a real constant in $\Omega=B_R$ and $n=1$ in $\R^2\setminus\overline{\Omega}$. Here $B_R$ is the ball of radius $R$ centered at the origin. Moreover, for the artificial background in (\ref{PbChampTotalFreeSpaceComp}), we take $\Om=\Omega$, $n_b=\rho / k^2$ in $\Omega$ and $n_b=1$ in $\R^2\setminus\overline{\Omega}$.\\
\newline
First, we compute the TEs of \eqref{NewITEP}. Using decomposition in Fourier series, one finds that if $w$, $v$ solve \eqref{NewITEP}, then $u=v+w$, $v$ admit the expansions 
\[
u(r, \theta)=\sum\limits_{m=-\infty}^{+\infty} b_m J_m\left(k\sqrt{n}r\right) e^{im\theta},\qquad\qquad
v(r, \theta)=\sum\limits_{m=-\infty}^{+\infty} c_m V_m\left(r\right) e^{im\theta},
\] 
where $b_m$, $c_m\in\mathbb{C}$, $J_m$ is the Bessel function of order $m$ and $V_m$ is defined by 
\[
V_m(r):= \left\{\begin{array}{cl}
\displaystyle{J_m\left(\sqrt{\rho}r\right) } & \text{ if $\rho\neq0$ }\\
\displaystyle{r^{|m|}} & \text{ if $\rho=0$.}
\end{array}\right.
\]
Note that we have
\[
 V'_m(r):= \left\{\begin{array}{cl}
\displaystyle{\sqrt{\rho}J'_m\left(\sqrt{\rho}r\right) } & \text{ if $\rho\neq0$ }\\
\displaystyle{|m|r^{|m|-1}} & \text{ if $\rho=0$.}
\end{array}\right.
\]
Imposing the condition $w=\partial_rw=0$ on $\partial B_R$, one obtains that $k$ is a TE of \eqref{NewITEP} if and only if there is some $m\in\mathbb{Z}$ such that
\begin{equation} \label{determinant}
\mbox{det}\left(
\begin{array}{cllrrr}
\displaystyle{V_m(R)} & \displaystyle{J_m(k\sqrt{n}R)} \\
\displaystyle{V'_m(R)} & \displaystyle{k\sqrt{n}J'_m(k\sqrt{n}R)}
\end{array}
 \right) =0.
 \end{equation}
Now we compute the eigenvalues of the modified farfield operator $\mathscr{F}$ defined in (\ref{ModifiedOp}). Using again decomposition in Fourier series, one gets that the functions $u$, $u_i$, $u_s$ solving \eqref{PbChampTotalFreeSpaceIntrod} expand as 
\begin{equation}\label{expansions}
\begin{array}{c}
u(r, \theta)=\dsp\sum_{m=-\infty}^{+\infty} b_mJ_m\left(k\sqrt{n}r\right) e^{im\theta}\mbox{ for }r\le R,\quad\dsp u_i(r, \theta)=\sum\limits_{m=-\infty}^{+\infty} c_m J_m\left(kr\right) e^{im\theta}\mbox{ for }r\le R,
 \\[12pt]
\dsp u_s(r, \theta)=\sum\limits_{m=-\infty}^{+\infty} d_m H^1_m\left(kr\right) e^{im\theta}\mbox{ for }r\ge R,
\end{array}\hspace{-0.2cm}
\end{equation}
where $b_m$, $c_m$, $d_m\in\mathbb{C}$ have to be determined. In the expansion of $u_s$, $H^1_m$ stands for the Hankel function of the first kind of order $m$. Imposing that $u-u_i|_{\partial B_R}=u_s|_{\partial B_R}$ and $\partial_r(u-u_i)|_{\partial B_R}=\partial_ru_s|_{\partial B_R}$, we obtain the system
\[
\left(
\begin{array}{cllrrr}
\displaystyle{J_m(k\sqrt{n}R)} & \displaystyle{-H^1_m(kR)} \\
\displaystyle{k\sqrt{n}J'_m(k\sqrt{n}R)}& \displaystyle{-kH'^{1}_m (kR)}
\end{array}
 \right) 
 \left(
\begin{array}{cllrrr}
\displaystyle{b_m}  \\
\displaystyle{d_m}
\end{array}
 \right) 
 =c_m
 \left(
\begin{array}{cllrrr}
\displaystyle{J_m(kR)}  \\
\displaystyle{kJ'_m(kR)}
\end{array}
 \right).
\]
Therefore $u$ and $u_s$ can be computed from $u_i$ using (\ref{expansions}) and the formula
\[
\begin{array}{cllrrr}
b_m & =c_m B_m \qquad\mbox{ with }\ B_m:=\dsp\frac{-H'^{1}_m (kR)J_m(kR)+H^1_m(kR)J'_m(kR)}{-J_m(k\sqrt{n}R)H'^{1}_m (kR)+\sqrt{n}J'_m(k\sqrt{n}R)H^1_m(kR)}, \\[14pt]
d_m & =c_m D_m \qquad\mbox{ with }\ D_m:=\dsp\frac{-\sqrt{n}J'_m(k\sqrt{n}R)J_m(kR)+J'_m(kR)J_m(k\sqrt{n}R)}{-J_m(k\sqrt{n}R)H'^{1}_m (kR)+\sqrt{n}J'_m(k\sqrt{n}R)H^1_m(kR)}.
\end{array}
\]
For the problem \eqref{PbChampTotalFreeSpaceComp} with artificial background, we can proceed to similar computations. The field $u_{b,s}$ admits a representation as $u_s$ in (\ref{expansions}) with some coefficients $d_{b,m}$ instead of $d_m$. Moreover $u_b$ expands as
\[
u_b(r, \theta)=\sum\limits_{m=-\infty}^{+\infty} b_{b,m}V_m\left(r\right)e^{im\theta}\mbox{ for }r\le R.
\]
Then we deduce that
\[
\begin{array}{cllrrr}
b_{b,m} & =c_m B_{b,m} \qquad\mbox{ with }\ B_{b,m}:=\dsp\frac{-kH'^{1}_m (kR)J_m(kR)+H^1_m(kR)kJ'_m(kR)}{-kV_m(R)H'^{1}_m (kR)+V'_m(R)H^1_m(kR)}, \\[14pt]
d_{b,m} & =c_m D_{b,m} \qquad\mbox{ with }\ D_{b,m}:=\dsp\frac{-V'_m(R)J_m(kR)+kJ'_m(kR)V_m(R)}{-kV_m(R)H'^{1}_m (kR)+V'_m(R)H^1_m(kR)}.
\end{array}
\]
If the incident field is the plane wave $u_i(x):=e^{ik\theta_i\cdot x}$, using Jacobi Anger formula, we obtain $c_m = i^m e^{-im\theta_i}$ . Moreover using that $H^1_m(kr)\underset{r\to+\infty}{\sim}(2/(\pi kr))^{1/2}e^{i(kr-m\pi/2-\pi/4)}$, one finds
\[
u^\infty_{s}(\theta_s,\theta_i)=\sqrt{\frac{2}{\pi k}}e^{-i\pi/4}\sum\limits_{m=-\infty}^{+\infty}e^{im(\theta_s-\theta_i)} D_m.
\]
For $u^\infty_{b,s}(\theta_s,\theta_i)$, we obtain a similar formula replacing $D_m$ with $D_{m,b}$. The far field pattern associated with an incident field coinciding with the Herglotz wave of density 
$$g(\theta)=\sum_{m=-\infty}^{+\infty} a_m e^{im\theta}$$
is then given by
\[
(Fg)(\theta_s)=\sqrt{\frac{8\pi}{k}}e^{-i\pi/4}\sum\limits_{m=-\infty}^{+\infty}a_me^{im\theta_s} D_m.
\]
And a similar expression holds for $F_b$. Thus we obtain analytic formulas of the far field operators $F$ and $F^b$.  One observes that $\theta \mapsto e^{im\theta}$ are the eigenfunctions of $F$, $F^b$ and that the corresponding eigenvalues are respectively given by
\[
\mu_m=\sqrt{\frac{8\pi}{k}}e^{-i\pi/4} D_m, \qquad\qquad
\mu_{b,m}= \sqrt{\frac{8\pi}{k}}e^{-i\pi/4} D_{b,m}.
\]
In 2D, in order to have unitary operators, $S$, $S^b$ are defined from $F$, $F^b$ by 
\[
S:=\mrm{Id}+2ik\frac{e^{-i\pi / 4}}{\sqrt{8\pi k}} F,\qquad\qquad S^b:=\mrm{Id}+2ik\frac{e^{-i\pi / 4}}{\sqrt{8\pi k}} F^b.
\]
Finally, we deduce that the eigenvalues of $\mathscr{S}=(S^b)^{\ast}S$ coincide with the set $\{\gamma_m;\,m\in\mathbb{Z}\}$ with
\[
\gamma_m:=(1-2ik\frac{e^{i\pi / 4}}{\sqrt{8\pi k}}\,\overline{\mu_{b,m}})(1+2ik\frac{e^{-i\pi / 4}}{\sqrt{8\pi k}}\,\mu_m)=(1+2\overline{D_{b,m}})(1+2D_m).
\]
We denote by $\hat{\delta}_m\in[0;2\pi)$ the phases of the $\gamma_m$. The $\delta_m$ introduced in (\ref{DefDelta}) then correspond to the ordered $\hat{\delta}_m$.\\
\newline
In Figure \ref{BckFig}, we take $n=2$ and $\rho=0$ (ZIM background) in $B_1$. We display the curves $k\mapsto \hat{\delta}_m(k)\in[0;2\pi)$ for $k\in(1;5.5)$. Each colour corresponds to a different value of $m\in\{0,\dots,300\}$. The vertical dotted lines represent the TEs of \eqref{NewITEP} computed by solving the determinant equation \eqref{determinant}. In $\Om$, we have $n-n_b=2>0$. And we see that the $\hat{\delta}_m$ accumulate only at $0$. This is coherent with the result of Proposition \ref{PropoOutsideTEs}. The black line represents the curve $k\mapsto \hat{\delta}_{\star}(k)$. In accordance with the statements of Theorems \ref{TheoPhase1} and \ref{TheoPhase2}, we observe that $\delta_{\star}$ tends to $2\pi$ as $k\nearrow k_0$ only when $k_0$ is TE of \eqref{NewITEP}.

\begin{figure}[!ht]
\centering
\includegraphics[width=9cm,trim={1cm 0.4cm 1cm 0.7cm},clip]{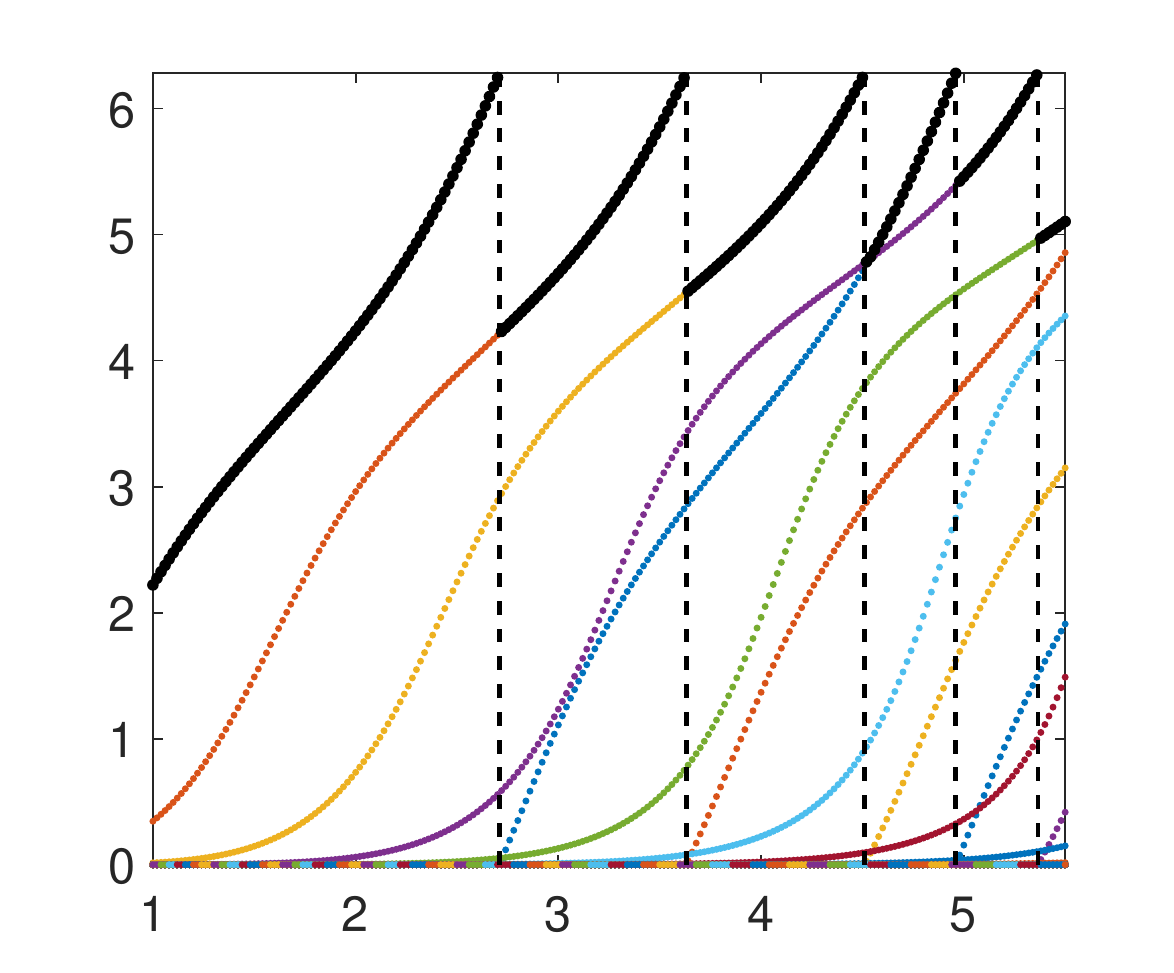}
\caption{We take $n=2$ and $\rho=0$ (ZIM background) in $B_1$ and we display the curves $k\mapsto \hat{\delta}_m(k)\in[0;2\pi)$ for $k\in(1;5.5)$. Each colour corresponds to a different value of $m\in\{0,\dots,300\}$.  The vertical dotted lines represent the TEs of \eqref{NewITEP} computed by solving the determinant equation \eqref{determinant}. The black line represents $k\mapsto \delta_{\star}(k)$.}
\label{BckFig}
\end{figure}

\noindent In Figures \ref{BckFigrho1}, \ref{BckFigrho1smallF}, we display similar curves in the case $n=2$ and $\rho=1$ in $B_1$. For Figure \ref{BckFigrho1}, we take $k\in(1/\sqrt{2};5.5)$. In this case, we have $n-n_b>0$ in $\Om$ and so the $\hat{\delta}_m$ accumulate only at $0$. For Figure \ref{BckFigrho1smallF}, we take $k\in(0;1/\sqrt{2})$. Then there holds $n-n_b<0$ in $\Om$ and the $\hat{\delta}_m$ accumulate only at $2\pi$. Note that for all these examples, we observe numerically that if $k_0$ is a TE of \eqref{NewITEP}, then the value of $m$ for which the determinant \eqref{determinant} is equal to zero coincides with the value of $p$ for which $k\mapsto \hat{\delta}_{p}(k)$ tends to $2\pi$ as $k\nearrow k_0$. This is coherent with the second part of the statement of Theorem \ref{TheoPhase1}.

\begin{figure}[!ht]
\centering
\includegraphics[width=9cm,trim={1cm 0.4cm 1cm 0.7cm},clip]{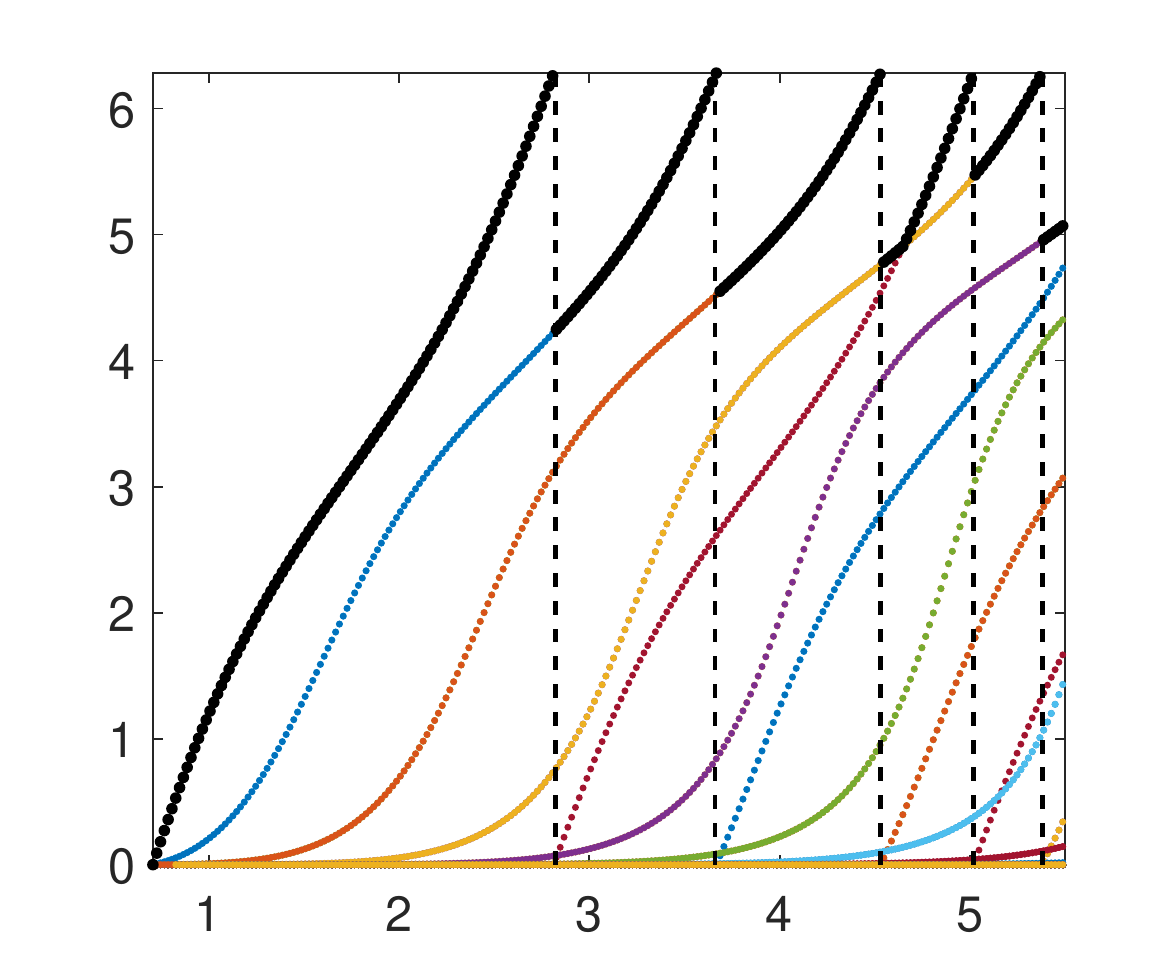}
\caption{We take $n=2$ and $\rho=1$ in $B_1$ and we display the curves $k\mapsto \hat{\delta}_m(k)\in[0;2\pi)$ for $k\in(1/\sqrt{2};5.5)$. Each colour corresponds to a different value of $m\in\{0,\dots,300\}$.  The vertical dotted lines represent the TEs of \eqref{NewITEP} computed by solving the determinant equation \eqref{determinant}. The black line represents $k\mapsto \delta_{\star}(k)$.}
\label{BckFigrho1}
\end{figure}

\begin{figure}[!ht]
\centering
\includegraphics[width=9cm,trim={1cm 0.4cm 1cm 0.7cm},clip]{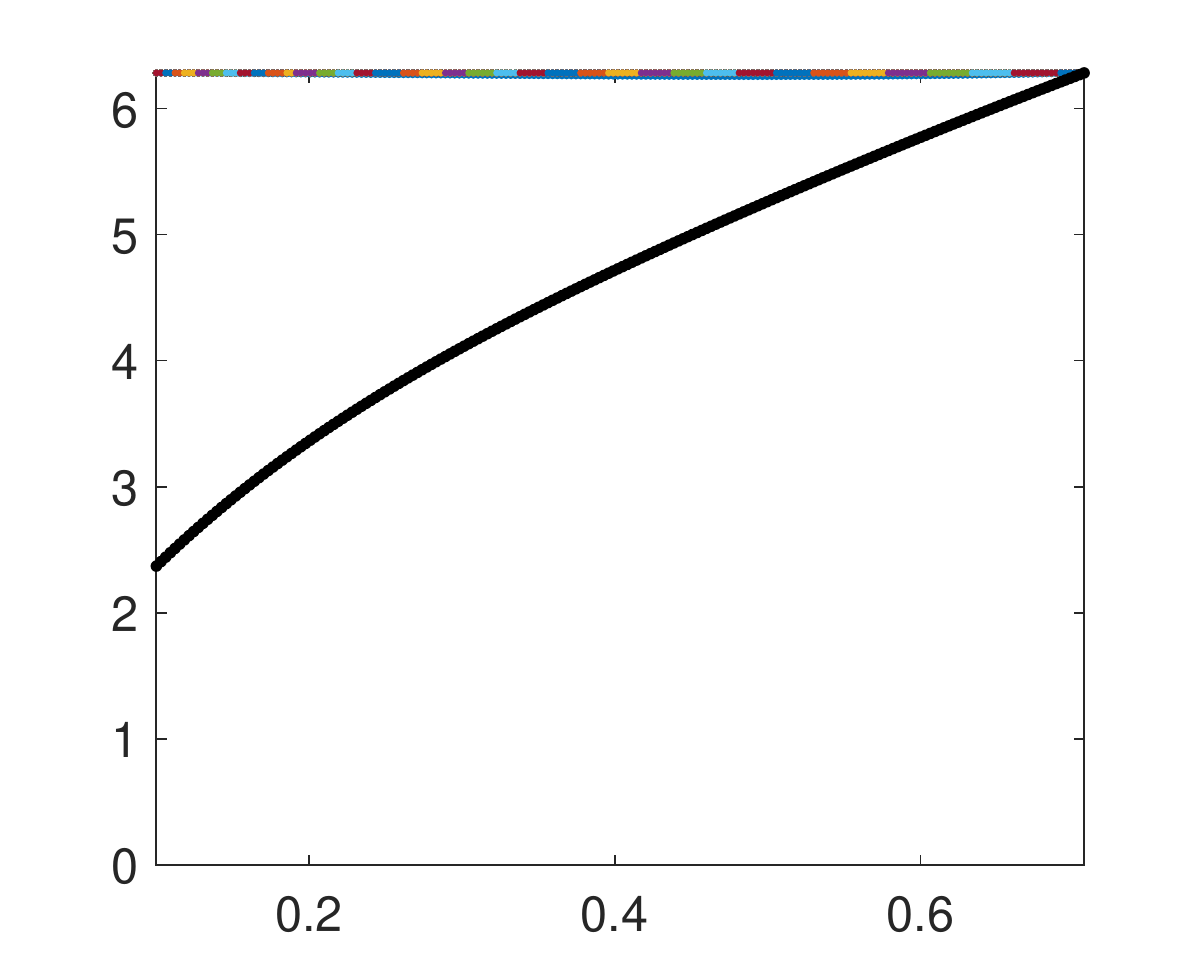}
\caption{We take $n=2$ and $\rho=1$ in $B_1$ and we display the curves $k\mapsto \hat{\delta}_m(k)\in[0;2\pi)$ for $k\in(0;1/\sqrt{2})$. Each colour corresponds to a different value of $m\in\{0,\dots,300\}$.  The black line represents $k\mapsto \delta_{\star}(k)$. Here there is no TE of \eqref{NewITEP} in this range of $k$ and as expected $\delta_{\star}(k)$ does not go to zero. }
\label{BckFigrho1smallF}
\end{figure}

\clearpage
\newpage

\section{Appendix}
In this Appendix, we prove Proposition \ref{PropoEstimRegu} ($v\mapsto w(k)$ is uniformly continuous from $\mL^2(\Om)$ to $\mH^2(\Om)$ in any compact set of $\R^{\ast}_+$) and Proposition \ref{ContinuousDependancek} ($H(k)$, $\mathscr{F}(k)$, $T(k)$ have continuous dependence with respect to $k$ in the operator norm).\\
\newline
\noindent\textit{Proof of Proposition \ref{ContinuousDependancek}. }
Introduce again some $R$ large enough so that $\overline{\Om}\subset B_R$. For a given $v\in\mL^2(\Om)$, denote $w(k)$ the solution of (\ref{eqwart}). 
For all $k\in\R^{\ast}_{+}$, the function $w(k)|_{B_R}$ satisfies the variational equality
\begin{equation}\label{VariaEqua}
(\nabla w(k),\nabla w')_{\Om}-k^2(n w(k),w')_{\Om}-\langle \Lambda(k)w(k),w'\rangle_{\partial B_R}=k^2((n-n_b(k))v,w')_{\Om},
\end{equation}
for all $w'\in\mH^1(B_R)$, where $\Lambda(k):\mH^{1/2}(\partial B_R)\to\mH^{-1/2}(\partial B_R)$ is the Dirichlet-to-Neumann operator defined in (\ref{DefDTN}). With the Riesz representation theorem, introduce the linear and bounded operator $\mathcal{T}(k):\mH^1(B_R)\to\mH^1(B_R)$ and the function $F(k)\in\mH^1(B_R)$ such that 
\[
\begin{array}{lcl}
(\mathcal{T}(k)\varphi,\varphi')_{\mH^1(B_R)}&=&(\nabla \varphi,\nabla \varphi')_{\Om}-k^2(n\varphi,\varphi')_{\Om}-\langle \Lambda(k)\varphi,\varphi'\rangle_{\partial B_R},\qquad\forall w,w'\in\mH^1(B_R)\\[3pt]
(F(k),\varphi')_{\mH^1(B_R)}&=&((n-n_b(k))v,\varphi')_{\Om},\qquad \forall w'\in\mH^1(B_R).
\end{array}
\]
With this notation, $w(k)$ solves (\ref{VariaEqua}) if and only if it satisfies
\[
\mathcal{T}(k)w(k)=k^2F(k).
\]
The Fredholm theory (injectivity holds thanks to the Rellich lemma) guarantees that $\mathcal{T}(k):\mH^1(B_R)\to\mH^1(B_R)$ is an isomorphism for all $k\in(0;+\infty)$. Using the explicit definition of $\Lambda(k)$, one can prove that the map $k\mapsto \Lambda(k)$ is continuous from $\R^{\ast}_{+}$ to $\mathcal{L}(\mH^{1/2}(\partial B_R),\mH^{-1/2}(\partial B_R))$. This allows us to show that $k\mapsto \mathcal{T}(k)$ is continuous
from $\R^{\ast}_{+}$ to $\mathcal{L}(\mH^1(B_R),\mH^1(B_R))$. Now, writing $\mathcal{T}(k)=\mathcal{T}(k_0)+(\mathcal{T}(k)-\mathcal{T}(k_0))$, we infer from the results on Neumann series that $k\mapsto \mathcal{T}(k)^{-1}$ is also continuous from $\R^{\ast}_{+}$ to $\mathcal{L}(\mH^1(B_R),\mH^1(B_R))$. As a consequence, $k\mapsto \mathcal{T}(k)^{-1}$ is uniformly bounded in any compact set $I\subset\R^{\ast}_{+}$. Since $k\mapsto F(k)$ is continuous, we have $\|F(k)\|_{\mH^1(B_R)}\le C\,\|v\|_{\mL^2(\Om)}$, where $C>0$ is independent from $k\in I$, $v$. We deduce the estimate $\|w(k)\|_{\mH^1(B_R)} \le C\,\|v\|_{\mL^2(\Om)}$ where $C$ is independent from $k\in I$, $v$. Finally, results of interior regularity leads to the estimate $
\|w(k)\|_{\mH^2(\Om)} \le C\,\|v\|_{\mL^2(\Om)}$ where $C$ is independent from $k\in I$, $v\in\mL^2(\Om)$.\hfill $\square$\\
\newline
\noindent\textit{Proof of Proposition \ref{ContinuousDependancek}. }Let us first show that $k\mapsto T(k)$ is continuous from $\R^{\ast}_{+}$ to $\mathcal{L}(\mL^2(\Om),\mL^2(\Om))$. We have 
\[
4\pi\,T(k)v=k^2(n - n_b(k)) (v + w(k)).
\]
From Proposition \ref{PropoEstimRegu} which guarantees that $v\mapsto w(k)$ is uniformly continuous from $\mL^2(\Om)$ to $\mH^2(\Om)$ for $k\in I$, $I$ being any compact set of $\R^{\ast}_+$, we deduce that $k\mapsto T(k)$ is continuous from $\R^{\ast}_{+}$ to $\mathcal{L}(\mL^2(\Om),\mL^2(\Om))$.\\
Now let us consider the continuity of $k\mapsto \herg(k)$ from $\R^{\ast}_{+}$ to $\mathcal{L}(\mL^2(\mathbb{S}^2),\mL^2(\Om))$. For all $g\in\mL^2(\mathbb{S}^2)$, we have $\herg(k) g=(u_i(k)+u_{b,s}(k))|_{\Om}$ where $u_{b,s}(k)$ is the scattered field of the solution of (\ref{PbChampTotalFreeSpaceComp}) with $u_i(k)=\int_{\mathbb{S}^2}g(\theta_{i})e^{ik\theta_{i}\cdot x}\,ds(\theta_{i})|_{\Om}$. Clearly $k\mapsto u_i(k)|_{\Om}$ is continuous from $\R^{\ast}_{+}$ to $\mathcal{L}(\mL^2(\mathbb{S}^2),\mL^2(\Om))$. And working exactly as for $k\mapsto T(k)$, one shows that $k\mapsto u_{b,s}(k)|_{\Om}$ is continuous from $\R^{\ast}_{+}$ to $\mathcal{L}(\mL^2(\mathbb{S}^2),\mL^2(\Om))$.\\ 
Finally, from the factorisation $\mathscr{F}(k)=H^{\ast}(k)T(k)H(k)$, we deduce that the mapping $k \mapsto \mathscr{F}(k)$ is continuous from $\R^{\ast}_{+}$ to $\mathcal{L}(\mL^2(\mathbb{S}^2),\mL^2(\mathbb{S}^2))$.\hfill $\square$

\section*{Acknowledgements}
This work was supported by a public grant as part of the
Investissement d'avenir project, reference ANR-11-LABX-0056-LMH,
LabEx LMH.

\bibliography{Bibli}
\bibliographystyle{plain}

\end{document}